\documentclass[reqno,11pt]{amsart}

\usepackage[top=2.0cm,bottom=2.0cm,left=1.9cm,right=1.9cm]{geometry}
\usepackage{amsthm,amsmath,amssymb,dsfont}
\usepackage{mathrsfs,amsfonts,functan,extarrows,mathtools}
\usepackage[colorlinks]{hyperref}

\usepackage{marginnote}
\usepackage{xcolor}
\usepackage{stmaryrd}
\usepackage{esint}
\usepackage{graphicx}
\usepackage{bm}

\newtheorem{theorem}{Theorem}[section]
\newtheorem{proposition}{Proposition}[section]

\newtheorem{corollary}{Corollary}[section]

\newtheorem{lemma}{Lemma}[section]

\numberwithin{equation}{section}
\allowdisplaybreaks

\def\R{\mathbb{R}}

\def\T{\mathbb{T}}
\def\P{\mathbf{P}}
\def\I{\mathbf{I}}
\def\v{\varepsilon}

\def\d{\textup{d}}

\def\l{\langle}
\def\r{\rangle}

\def\up{\textup}
\def\n{\nabla}
\def\p{\partial}

\def\a{\alpha}
\def\b{\beta}

\def\D{\mathcal{D}}
\def\c{\cdot}

\def\E{\mathcal{E}}
\def\M{\sqrt{M}}

\newcounter{wronumber}

\begin{document}

\title[Hydrodynamic limit of VPFP system in low-field regime]
{Hydrodynamic limit of the Vlasov-Poisson-Fokker-Planck system in low-field regime}

\author[Zhendong Fang]{Zhendong Fang}
\address[Zhendong Fang]
        {\newline School of Mathematics and Information Science, Guangzhou University, Guangzhou 510006, P. R. China}
\email{zdfang@gzhu.edu.cn}

\author[Kunlun Qi]{Kunlun Qi}
\address[Kunlun Qi]
        {\newline Simons Laufer Mathematical Sciences Institute (former MSRI), Berkeley, CA 94720, USA
        \newline Department of Computational Mathematics, Science and Engineering and Department of Mathematics, Michigan State University, East Lansing, MI 48824, USA }
\email{kunlunqi.math@gmail.com}



%\thanks{The authors would thank...}

\begin{abstract}
In this paper, we study the hydrodynamic limit of the scaled Vlasov–Poisson–Fokker–Planck (VPFP) system in the low-field regime. By employing the moment method, we formally derive the corresponding Drift–Diffusion–Poisson (DDP) system. Furthermore, we rigorously justify the pointwise convergence from the VPFP system to the DDP system through delicate high-order energy estimates based on a Macro–Micro decomposition. The main difficulty lies in controlling the nonlinear coupling between the kinetic and electrostatic fields and establishing uniform bounds with respect to the scaling parameter. These challenges are overcome by developing refined high-order energy methods that yield uniform energy estimates and ensure the global well-posedness of smooth solutions, without relying on any a \textit{priori} assumptions for the limiting DDP system.
\end{abstract}

\maketitle

%\vspace*{10pt}

\phantomsection
\addcontentsline{toc}{section}{\contentsname}

%\tableofcontents %编写目录的指令

%%%%%%%%%%%%%%%%%%%%%%%%%%%%%%%%%%%%%%（正文）%%%%%%%%%%%%%%%%%%%%%%%%%%%%%
%%%%%%%%%%%%%%%%%%%%%%%%%%%%%%%%%%%%%%%%%%%%%%%%%%%%%%%%%%%%%%%%%%%%%%%%%%%

\section{Introduction}
\label{sec:intro}

\subsection{The model}
\label{subsec:bg}

We consider the hydrodynamic limit of a scaled Vlasov–Poisson–Fokker–Planck (VPFP) system, originally proposed in \cite{BostanGoudon2008}, which describes the collective behavior of a large number of charged particles under the combined effects of self-consistent electrostatic interactions and diffusion. The scaled VPFP system takes the form:
\begin{equation}\label{eq:Sclaed-VPFP-0}
\left\{
\begin{aligned}
&\p_t f^\v+\frac{1}{\v}v\c\n_x f^\v-\frac{1}{\v}\n_x\phi^\v\c\n_v f^\v=\frac{1}{\v^2}\up{div}_v(\n_v f^\v+v f^\v),\\[5pt]
&-\Delta_x\phi^\v=\rho^\v-1, \quad \rho^\v = \int_{\R^3}f^\v \, \d v\\[5pt]
&f^\v(t=0,x,v) = f^{\v,in}(x,v),\\[5pt]
\end{aligned}
\right.
\end{equation}
where $f^\v :=f^\v(t, x, v)$ is the distribution function of charged particles at time $t \geq 0$ in position $x\in\T^3$ with velocity $v\in\R^3$, $\rho^\v:= \rho^\v(t,x)$ denotes the macroscopic electron density,
and $\phi^\v:=\phi^\v(t,x)$ is the self-consistent electrostatic potential determined by the Poisson equation with the constant background charge density normalized to one for simplicity.
In addition, here $\v$ is a dimensionless parameter related to the mean free path, for which we refer the readers to \cite{BostanGoudon2008} for more physical intuition. 

In this paper, we rigorously justify the diffusion limit of the scaled VPFP system \eqref{eq:Sclaed-VPFP-0} in the sense that: as $\v \to 0$, the solution of the scaled VPFP system \eqref{eq:Sclaed-VPFP-0} converges to the solution of the following macroscopic Drift-Diffusion-Poisson (DDP) system:
\begin{equation}\label{The-Drift-Diffusion-Possion-system}
\left\{
\begin{aligned}
&\p_t\rho=\Delta_x\rho+\up{div}_x(\rho\n_x\phi),\\[4pt]
&-\Delta_x\phi=\rho-1,\\[4pt]
&\rho(0,x)=\rho^{in}(x).\\[4pt]
\end{aligned}
\right.
\end{equation}
The DDP system \eqref{The-Drift-Diffusion-Possion-system} provides a macroscopic description of charge transport, where the evolution of the density $\rho$ is governed by the combined effects of diffusion and drift under the self-consistent electrostatic potential $\phi$.

%%%%%%%%%%%%%%%%%%%%%%%%%%%%%%%%%%%%%%%
\subsection{Previous results and our contributions}
\label{subsec:previous}

In view of its fundamental physical importance, the VPFP system has been extensively investigated for a long history. We begin by reviewing the existing literature on its well-posedness and hydrodynamic limits, followed by a discussion that emphasizes our main contributions and novelties of this work through comparison with previous results.\\[-10pt]

\textit{Previous results for ``well-posedness"}: 

The well-posedness theory of the VPFP system and related models has been studied over the past decades. In \cite{Degond1986}, Degond established the global existence of smooth solutions in one and two dimensions and the local existence in three dimensions for the Vlasov–Fokker–Planck equation. Later, Victory-O'Dwyer \cite{Victory-Harold-Dwyer1990} proved the global existence of classical solutions for arbitrary initial data when the spatial or momentum dimension is less than or equal to two, and obtained local existence results for arbitrary data in higher dimensions. Furthermore, Bouchut \cite{Bouchut1993} conducted a detailed analysis of the regularity properties of solutions to the linear Vlasov–Fokker–Planck equation with a force field, and established the existence and uniqueness of smooth solutions to the three-dimensional VPFP system. The smoothing effect for the nonlinear three-dimensional VPFP system was subsequently investigated in \cite{Bouchut1995}.
In addition, Carpio \cite{Carpio1998} studied the long-time behavior of solutions to the VPFP system with sufficiently small initial data under suitable integrability assumptions. Hwang-Jang \cite{Hwang-Jang2013} later proved the exponential decay in time of small-amplitude smooth solutions near a global Maxwellian equilibrium, both in the whole space and in the periodic domain, by employing uniform-in-time energy estimates. More recently, Tan-Fan \cite{Tan-Fan2024} established the global-in-time existence of mild solutions to the VPFP system near a global Maxwellian equilibrium, again under small-amplitude initial perturbations.
In addition to the Fokker–Planck framework, it is worth mentioning the well-posedness theory for the Vlasov–Poisson equation coupled with the Boltzmann (VPB) and Landau (VPL) collision operators. For the VPB system, Guo \cite{Guo2002} established the unique global-in-time classical solution by developing an energy method incorporating a new dissipation estimate for the collision term. Later, Yang-Zhao \cite{Yang-Zhao2006} obtained global classical solutions for small initial perturbations by combining the theory of compressible Navier–Stokes equations with a refined Macro–Micro decomposition.
For the VPL system, Guo \cite{Guo2012} proved the existence of unique global solutions near Maxwellians via nonlinear energy methods and derived decay estimates through a bootstrap argument. More recently, Dong-Guo-Ouyang \cite{Dong-Guo-Ouyang2022} established global stability and well-posedness near Maxwellians with time decay, introducing new regularity estimates and an improved 
$L^2$ to $L^\infty$ energy framework for the VPL system under specular reflection boundary conditions.
For additional well-posedness results on Vlasov–Poisson equations coupled with other kinetic models, we refer the reader to \cite{Dong-Guo-Ouyang-Yastrzhembskiy2024, Dong-Yang-Zhong2019, Dong-Yastrzhembskiy2024, Duan-Yu2020, Jiang-Lei-Zhao2024, Li-Yang-Zhong2016, Xiao-Xiong-Zhao2017, Yang-Yu-Zhao2006}
For studies concerning the Vlasov–Fokker–Planck equation coupled with other physical models, we refer the reader to \cite{CJADRJMA11, CMKKLJ11, Dong-Yastrzhembskiy2022, GHMAZP10,  LFCMYMWDH17, LHLLSQYT22, MAVA07, MYMWDH20, YFSLY20} and the references therein.\\[-10pt]

\textit{Previous results for ``hydrodynamical limits"}: 

Based on the well-posedness, the hydrodynamic limit of the VPFP system has also been investigated from several perspectives. 
Early works centered on establishing weak convergence: Poupaud-Soler \cite{Poupaud-Soler2000} studied the parabolic limit, proving a weak $L^1$ convergence result, along with the stability of its solutions. This was subsequently extended by El Ghani-Masmoudi \cite{El-Masmoudi2010} to general spatial dimensions $d \geq 2$. Goudon \cite{G2005} also established global-in-time convergence in a weak $L^1$ framework for the 2D system with general initial data. We refer the more results about the hydrodynamical limit in the weak sense to \cite{GoudonNietoPoupaudSoler2005,NPS2001,WLL2015}.
More recently, attention has shifted to strong convergence and convergence near equilibrium: Zhong \cite{Zhong2022} proved the convergence of strong global solutions near the global Maxwellian by spectral analysis, establishing the optimal convergence rate and providing precise estimates for the initial layer. Blaustein \cite{B2023} established a strong convergence result for the diffusive limit in a low-regularity $L^p$ setting (for sufficiently large $p$). In the one-dimensional case, Lehman-Negulescu \cite{Lehman-Negulescu2025} also recently demonstrated strong $L^2$ convergence for the asymptotic limit.
For additional developments concerning the hydrodynamic limit of the VPFP system and related kinetic models, we refer the reader to \cite{Addala-Dolbeault-Li-Tayeb2021,ACGS2001, El2010, FangQiWen2024, Herda-Rodrigues2018} and the references therein.\\[-10pt]

%-------------------------
\textit{Mathematical challenges and our contributions}:  

In contrast to previous works on strong convergence -- such as the $L^p$ framework in \cite{B2023, Lehman-Negulescu2025} and the $H^2_x L^2_v$ framework in \cite{Zhong2022} -- the present paper provides a rigorous justification of the hydrodynamic limit from the VPFP system \eqref{The-VPFP-system} to the DDP system \eqref{The-Drift-Diffusion-Possion-system-2} in a stronger pointwise sense with respect to both velocity and spatial variables (see Theorem \ref{Limit-Fluid-equations}).
The key novelty of our approach lies in the fact that we do not assume the existence of solutions to the limiting DDP system a \textit{priori}. Instead, we derive uniform-in-$\v$ estimates directly for the scaled VPFP system \eqref{eq:Sclaed-VPFP-0} (see Theorem \ref{Global-in-time-solution-of-VPFP}). Through a compactness argument, we obtain the DDP system as the limiting dynamics of the VPFP solutions, thereby establishing its existence simultaneously with the limit process.

% In contrast to the previous work on the strong convergence such as the $L^p$ framework in \cite{B2023, Lehman-Negulescu2025} and $H^2_x L^2_v$ framework in \cite{Zhong2022}, we further provide a rigorous justification of the hydrodynamic limit from the VPFP system \eqref{The-VPFP-system} to the DDP system \eqref{The-Drift-Diffusion-Possion-system-2} in a stronger pointwise convergence in both velocity and space variable. Notably, rather than relying on the existence of solutions to the limiting DDP system, we directly perform the uniform estimates for the scaled VPFP system \eqref{eq:Sclaed-VPFP-0}. Then, by employing compactness arguments, we derive the DDP system directly from the solutions of the VPFP system as well as establishing its existence. 

To this end, the main analytical challenge arises from obtaining uniform high-order energy estimates for the scaled VPFP system \eqref{The-VPFP-system}. In the classical energy method, it is essential to prove global well-posedness of the DDP system \eqref{The-Drift-Diffusion-Possion-system-2}, and this typically relies on delicate higher-order estimates for the DDP equations themselves, which is highly nontrivial. In this work, we circumvent the need for any a \textit{priori} estimates on the limiting system by constructing and controlling intricate higher-order energy functionals for the VPFP system in the kinetic regime. These uniform estimates not only yield the hydrodynamic limit but also ensure the global well-posedness of the DDP system as a direct consequence of the limiting process.

% The energy method is an indispensable technique for establishing the global well-posedness of the DDP system \eqref{The-Drift-Diffusion-Possion-system-2}, and it typically relies on the higher-order energy estimates for the DDP system, which is a highly challenging task. In this paper, we start directly from the global solution of the scaled VPFP system \eqref{The-VPFP-system} and, by employing uniform energy estimates for the scaled system, we deduce the global well-posedness of solutions to the DDP system \eqref{The-Drift-Diffusion-Possion-system-2}. It is worth noting that, in this approach, no a \textit{priori} estimates is required for DDP system \eqref{The-Drift-Diffusion-Possion-system-2}.

Meanwhile, to establish the pointwise convergence, we design refined energy–dissipation structures (see \eqref{Def_total_functional}) based on the classical Macro–Micro decomposition, which aligns naturally with our perturbation form \eqref{def:perturbation-form}. These functionals incorporate high regularity in both the spatial and velocity variables, allowing the pointwise convergence to directly follow from the total higher-order energy estimate (see Proposition \ref{total_energy_dissipation}) via standard embedding arguments (see Corollary \ref{Conergence}). 
To the best of our knowledge, this work provides the first rigorous justification of the hydrodynamic limit from the scaled VPFP system to the DDP system in a strong pointwise sense. This result relies on a delicate and technically demanding hierarchy of high-order uniform energy estimates. Beyond the VPFP setting, the framework developed here has the potential to be extended to the hydrodynamic limits of other coupled kinetic–field systems, such as the Vlasov–Maxwell–Fokker–Planck equations, thereby offering a unified approach for treating hydrodynamic limits in more complex models.

%Specifically, we establish the pointwise convergence of $g^\v$ with respect to both $x$ and $v$, and of $\n_x\phi^\v$ with respect to $t$ and $x$ .

%Furthermore, to achieve the pointwise convergence $(g^\v, \n_x\phi^\v)$, we construct the delicate energy/dissipation structures (see \eqref{Def_total_functional}). The energy/dissipation structures are constructed based on the well-known Macro-Micro decomposition, which perfectly matches our expansion form \eqref{def:perturbation-form}. In addition, the energy/dissipation functionals also involve the high regularity in both spatial variable $x$ and velocity variable $v$, such that the pointwise convergence can directly follow from the total energy estimate (Proposition \ref{total_energy_dissipation}) by the embedding theorem. Specifically, we prove the pointwise convergence of $g^\v$ with respect to the position $x$ and velocity $v$, as well as the pointwise convergence of $\n_x\phi^\v$ with respect to the time $t$ and space $x$. To our best knowledge, this should be the first result to rigorously justify the hydrodynamic limits, specifically from the scaled VPFP system \eqref{The-VPFP-system} to the DDP system \eqref{The-Drift-Diffusion-Possion-system-2}, in the stronger pointwise sense, thanks to the delicate high-order energy estimate. 

%the key point is to derive a global-in-time energy bound for the remainder system \eqref{Remainder equations} uniformly $\v\in(0,1]$

%------------------------------------
\subsection{Organization of the paper}
\label{subsec:organization}
This paper is organized as follows: The main results are first presented in the following Section \ref{sec:main}. In Section \ref{sec:formal}, the formal derivation is shown by the moment method. We then develop the global-in-time energy estimate of the scaled VPFP system in Section \ref{sec:energy}. Based on the delicate energy estimate above, the strong pointwise convergence is finally justified in Section \ref{sec:limit}. 

%%%%%%%%%%%%%%%%%%%%%%%%%%%%%%%%%%%%%%%%%%%%%%%%%%%%%%%%%%%%%%%%%%%%%%%%
\section{Notations and main results}
\label{sec:main}

\subsection{Notations}
\label{subsec:notation}
Before stating the main result, we give some notations in this paper.
\begin{itemize}
\item $A \lesssim B \Leftrightarrow A \leq C B$, $A \sim B\Leftrightarrow C_1 A \leq B \leq C_2 A$, for some generic constants $C,C_1,C_2>0$ independent of $t$ and $\v$.

\item For multi-indices $\alpha = (\alpha_1, \alpha_2, \alpha_3)$ and $\beta = (\beta_1, \beta_2, \beta_3)$, we denote
\begin{equation*}
    \partial_x^\alpha = \partial_{x_1}^{\alpha_1} \partial_{x_2}^{\alpha_2} \partial_{x_3}^{\alpha_3}, \quad \partial_v^\beta = \partial_{v_1}^{\beta_1} \partial_{v_2}^{\beta_2} \partial_{v_3}^{\beta_3}.
\end{equation*}

\item For the multi-indices $\alpha$ and $\alpha'$, we denote $\binom{\a}{\a'}$ as  the binomial coefficient.

\item Let $\nu=1+|v|^2$ and we denote $\|\c\|_{\nu}$ by
\begin{equation*}
\|g\|_{\nu}=\left( \int_{\T^3}\int_{\R^3} |\n_v g|^2+|g|^2\nu \, \d v \, \d x \right)^{\frac{1}{2}}.
\end{equation*}

\item For $d,e\in\mathbb{N}$, we denote the following inner-products with associated norms:
\begin{equation*}
\begin{aligned}
\l u,w\r_x =& \int_{\T^3} uw \,\d x,\,  \quad  \l f,g\r_v=\int_{\R^3} fg \,\d v,\, \quad \|u\|_{L^2_x}= \l u,u \r_x^{\frac{1}{2}},\, \quad \|f\|_{L^2_v}= \l f,f\r_v^{\frac{1}{2}}, \\[4pt]
\l f,g \r_{x,v} =& \int_{\T^3} \int_{\R^3} fg \,\d v \,\d x, \quad \|f\|_{L^2_{x,v}}=\l f,f\r_{x,v}^{\frac{1}{2}},\\
\end{aligned}
\end{equation*}
and the function spaces:
\begin{equation*}
\begin{aligned}
H^d_{x}:=&\big\{u(x) \ \big| \ \|\p_x^\alpha u\|_{L^2_x} <\infty,\,\up{for any}\,|\a|\leq d\big\},\\[6pt]
H^d_{x,v}:=&\big\{f(x,v) \ \big| \ \|\p_x^\alpha\p_v^\beta f\|_{L^2_{x,v}} <\infty,\,\up{for any}\,|\a|+|\b|\leq d\big\},\\[6pt]
\mathcal{H}^d_{x,v}:=&\big\{f(x,v) \ \big| \ \|\p_x^\alpha\p_v^{\beta} f\|_{\nu} <\infty,\,\up{for any}\,|\a|+|\b|\leq d\big\},\\[6pt]
H^d_x H^e_v:=&\big\{f(x,v) \ \big| \ \|\p_x^\alpha\p_v^\beta f\|_{L^2_{x,v}} <\infty,\,\up{for any}\,|\a|\leq d,\,|\b|\leq e\big\},\\[6pt]
\mathcal{H}^d_x\mathcal{L}^2_v : = &\big\{f(x,v) \ \big| \ \|\p_x^\alpha  f\|_{\nu} <\infty,\,\up{for any}\,|\a|\leq d\big\}\\[6pt]
\mathcal{H}^d_x \mathcal{H}^e_v:=&\big\{f(x,v) \ \big| \ \|\p_x^\alpha\p_v^\beta f\|_{\nu} <\infty,\,\up{for any}\,|\a|\leq d,\,|\b|\leq e\big\}.
\end{aligned}
\end{equation*}
\item We denote $M$ as a global normalized Maxwellian equilibrium given by
\begin{equation}\label{Maxwellian}
M:=M(v)=\frac{1}{(2\pi)^{\frac{3}{2}}}\up{e}^{-\frac{|v|^2}{2}}.
\end{equation}

\end{itemize}

%%%%%%%%%%%%%%%%%%%%%%%%%%%%%%%%%%%%%%%%%%%%
\subsection{Main results}
\label{subsec:main}

We consider the solution around the global Maxwellian distribution $M$ in the sense that
\begin{equation}\label{def:perturbation-form}
    f^\v(t,x,v) = M + g^\v(t,x,v) \sqrt{M},
\end{equation}
by substituting \eqref{def:perturbation-form} into \eqref{eq:Sclaed-VPFP-0}, we obtain the scaled VPFP system for $(g^\v,\n_x\phi^\v)$: 
\begin{equation}\label{The-VPFP-system}
\left\{
\begin{aligned}
&\p_t g^\v+\frac{1}{\v}v\c\n_x g^\v+\frac{1}{\v}v\c\n_x\phi^\v\M+\frac{1}{\v}
\left(\frac{g^\v}{2}v\c\n_x\phi^\v-\n_v g^\v\c\n_x\phi^\v\right)+\frac{1}{\v^2}Lg^\v=0,\\[4pt]
&-\Delta_x\phi^\v=a^\v,\\[4pt]
& g^\v(0,x,v)=g^{\v,in}(x,v),
\end{aligned}
\right.
\end{equation}
where $L$ is the Fokker-Planck operator
\begin{equation}\label{def-L}
  L g^\v:=-\frac{1}{\M}\up{div}_v\left(M\n_v \left(\frac{g^\v}{\M} \right)\right),  
\end{equation}
and $a^\v$ is given by
\begin{equation}\label{The-def-a}
    a^\v: = a^\v(t,x)=\int_{\R^3}g^\v\M\d v.
\end{equation}

In the following Theorem \ref{Global-in-time-solution-of-VPFP}, we present the global well-posedness of the scaled VPFP system \eqref{The-VPFP-system} with the corresponding total energy estimate.

\begin{theorem}\label{Global-in-time-solution-of-VPFP}
For any integer $k\geq3$, there exists a small constant $\delta_0>0$ such that, if $\mathbb{E}_k(0)\leq\delta_0$, then the scaled VPFP system \eqref{The-VPFP-system} admits a unique solution $(g^\v,\n_x\phi^\v)$ satisfying
\begin{equation}\label{The-funcutions-of-space}
\begin{aligned}
  &g^\v (t,x,v) \in L^\infty \left(0,+\infty;H^k_{x,v}\right),\quad (\I-\P_0)g^\v(t,x,v) \in L^2\left(0,+\infty;\mathcal{H}^k_{x,v}\right),\\[5pt]
  &\n_x\phi^\v(t,x) \in L^\infty\left(0,+\infty;H^k_x\right)\cap L^2\left(0,+\infty;H^k_x\right)
\end{aligned}
\end{equation}
with uniform energy estimate
\begin{equation}\label{Uniform energy estimate}
\sup_{t\geq 0}\mathbb{E}_k(t) + \tilde{C} \int_0^{+\infty}\mathbb{D}_k(\tau) \, \d\tau \lesssim \mathbb{E}_k(0),
\end{equation}
where the energy functional $\mathbb{E}_k$ and dissipation functional $\mathbb{D}_k$ are defined in \eqref{Def_total_functional} and $\tilde{C}$ is independent of $\v$.
\end{theorem}

Based on the well-posedness and uniform energy estimate in Theorem \ref{Global-in-time-solution-of-VPFP} above, we can rigorously justify the hydrodynamical limit from the scaled VPFP system to the DDP system in Theorem \ref{Limit-Fluid-equations}. More specifically, as $\v \to 0$, the solution $(g^\v, \nabla_x\phi^\v)$ to \eqref{The-VPFP-system} can be shown to converge to $(\rho_0 \M, \nabla_x\phi_0)$, which are the solutions to the following DDP system \eqref{The-Drift-Diffusion-Possion-system-2}:
\begin{equation}\label{The-Drift-Diffusion-Possion-system-2}
\left\{
\begin{aligned}
&\p_t\rho_0=\Delta_x\rho_0+\up{div}_x\big[(\rho_0+1)\n_x\phi_0\big],\\[4pt]
&-\Delta_x\phi_0=\rho_0,\\[4pt]
\end{aligned}
\right.
\end{equation}
where we denote $\rho_0:=\rho-1$ and $\phi_0:=\phi$ with $\rho,\phi$ being the solutions to the original DPP system \eqref{The-Drift-Diffusion-Possion-system}.

\begin{theorem}\label{Limit-Fluid-equations}
Under the conditions of Theorem \ref{Global-in-time-solution-of-VPFP}, let $(g^{\v,in},\n_x\phi^{\v,in})$ be the initial condition satisfying
\begin{equation}\label{The-initial-datas}
    \begin{aligned}
     &g^{\v,in}(x,v)\to\,\rho_0^{in}(x) \M, \quad\up{strongly in $H^k_{x,v}$},\\[5pt]
     &\n_x\phi^{\v,in}(x)\to\,\n_x\phi_0^{in}(x),\quad\up{strongly in $H^k_x$},
    \end{aligned}
\end{equation}
as $\v\to\,0$, and $(g^\v,\n_x\phi^\v)$ be a sequence of solutions to the scaled VPFP system \eqref{The-VPFP-system} obtained by Theorem \ref{Global-in-time-solution-of-VPFP}. Then, for any given $T>0$,
\begin{equation}\label{The-functions-datas}
    \begin{aligned}
    &g^\v(t,x,v)\to\,\rho_0(t,x)\M,\quad\up{weakly-$\star$ in $t\in[0,T]$, strongly in $H^{k-1}_x$, weakly in $H^k_v$},\\[5pt]
     &\n_x\phi^\v(t,x)\to\,\n_x\phi_0(t,x),\quad\up{weakly-$\star$ in $t\in[0,T]$, strongly in $H^{k-1}_x$},
    \end{aligned}
\end{equation}
as $\v\to\,0$, where 
\[
\rho_0\in L^\infty \left(0,T;H^k_x\right)\cap C\left([0,T];H^{k-1}_x\right),\quad \n_x\phi_0\in L^\infty\left(0,T;H^k_x\right)\cap C\left([0,T];H^k_x\right)
\]
is the unique solution to the DPP system \eqref{The-Drift-Diffusion-Possion-system-2} with the initial conditions $(\rho_0^{in}, \n_x\phi^{in}_0)$ given in \eqref{The-initial-datas}.
% \begin{equation}
%    \rho_0(0,x)=\rho_0^{in}(x),\quad \n_x\phi_0(0,x)=\n_x\phi^{in}_0(x).
% \end{equation}

Furthermore, the convergence of the moments holds:
\begin{equation}
    \begin{aligned}
     &\l g^\v(t,x,\cdot),\M \r_v\to\,\rho_0(t,x),\quad\up{strongly in $C\left([0,T];H^{k-1}_x\right)$},\\[5pt]
     &\n_x\phi^\v(t,x) \to\,\n_x\phi_0(t,x),\quad\,\,\,\,\,\up{strongly in $C\left([0,T];H^{k-1}_x\right)$},   
    \end{aligned}
\end{equation}
as $\v\to\,0$.

\end{theorem}

Thanks to Theorem \ref{Global-in-time-solution-of-VPFP}, the following Corollary \ref{Conergence} on pointwise convergence is directly obtained by the Sobolev embedding $H^2\hookrightarrow L^\infty$, and the complete proof is given in Section \ref{subsec:proof_main_corollary}.

\begin{corollary}\label{Conergence}
Under the conditions of Theorem \ref{Limit-Fluid-equations} with $k \geq 4$,  the following pointwise convergence holds: for any given $T>0$,
\begin{equation}
\begin{aligned}
\lim_{\v \rightarrow 0 }\int_0^T \left|f^\v(t,x,v)-\big[1+\rho_0(t,x) \big] M\right|^2 \,\d t & = 0,\\[4pt]
\lim_{\v \rightarrow 0 } \left|\n_x\phi^\v(t,x)-\n_x\phi_0(t,x) \right|&= 0,
\end{aligned}
\end{equation} 
for $(t,x,v)\in [0,T]\times\T^3\times\R^3$.
\end{corollary}

%%%%%%%%%%%%%%%%%%%%%%%%%%%%%%%%%%%%%%%%%%%%%%%%%%%%%%%%%%%%%%%%%%%%%%%%

%\section{Main result}
%\label{sec:notation_main}

%%%%%%%%%%%%%%%%%%%%%%%%%%%%%%%%%%%%%%%%%%%%%%%%%%%%%%%%%%%%%%%%%%%%%%%%%%%%%%%%%%%%%%%%%%%%%%%%%%
\section{Formal analysis via moment method}
\label{sec:formal}

In this section, we employ the moment method to perform a formal asymptotic derivation with two objectives: first, to obtain the corresponding macroscopic system through moment closure; and second, to clarify the analytical framework that serves as the foundation for the subsequent rigorous proof.

\textbf{Step 1}: We start with re-writing the scaled VPFP system \eqref{eq:Sclaed-VPFP-0} as follows:
\begin{equation}\label{eq:scaled-VPFP-1}
\left\{
\begin{aligned}
&\up{div}_v(\n_v f^\v+v f^\v)=\v^2\p_t f^\v+\v v\c\n_x f^\v-\v\n_x\phi^\v\c\n_v f^\v,\\[4pt]
&-\Delta_x\phi^\v=\rho^\v-1.\\[4pt]
\end{aligned}
\right.
\end{equation}

Suppose that
\begin{equation}\label{Formal-form}
\begin{aligned}
f^\v\to f_0,\quad \phi^\v\to \phi_0, \quad \text{as} \quad \v\to 0,
\end{aligned}
\end{equation}
when taking $\v\to 0$ in \eqref{eq:scaled-VPFP-1}, the right-hand side of $\eqref{eq:scaled-VPFP-1}_1$ vanishes, and it yields
\begin{equation*}
\left\{
\begin{aligned}
&\up{div}_v(\n_v f_0+v f_0)=\up{div}_v\left[M\n_v\left(\frac{f_0}{M}\right)\right]=0,\\[4pt]
&-\Delta_x\phi_0 = \int_{\R^3} f_0 \,\d v-1,\\[4pt]
\end{aligned}
\right.
\end{equation*}
 which further implies that 
\begin{equation}\label{eq:f0}
f_0(t,x,v)=\rho(t,x) M,    
\end{equation}
where $\rho(t,x)$ is the function to be determined, and $M$ is the Maxwellian distribution as in \eqref{Maxwellian}. 

\textbf{Step 2:} To further determine the macroscopic equation satisfied by $\rho(t,x)$, we need to rely on the properties of the self-adjoint Fokker–Planck operator $L$. Specifically, recalling \eqref{def:perturbation-form} and noting \eqref{Formal-form}-\eqref{eq:f0}, we have
\begin{equation}\label{eq:convergence-g-v}
    g^\v(t,x,v) = \frac{f^\v(t,x,v) - M}{\M} \to \big( \rho(t,x)-1 \big) \M, \quad  \text{as} \quad \v\to 0.
\end{equation}
Furthermore, by substituting \eqref{def:perturbation-form} into \eqref{eq:scaled-VPFP-1}, we have  
\begin{equation}\label{eq:scaled-VPFP-2}
\left\{
\begin{aligned}
&\p_t g^\v+\frac{1}{\v}v\c\n_x g^\v+\frac{1}{\v}v\c\n_x\phi^\v\M+\frac{1}{\v}\left(\frac{g^\v}{2}v\c\n_x\phi^\v-\n_v g^\v\c\n_x\phi^\v\right)+\frac{1}{\v^2}Lg^\v=0,\\[4pt]
&-\Delta_x\phi^\v=a^\v,
\end{aligned}
\right.
\end{equation}
where $L$ is the Fokker-Planck operator as in \eqref{def-L} and $a^\v$ is defined in \eqref{The-def-a}.

In addition, according to $\eqref{eq:scaled-VPFP-2}_2$, we can obtain the following Poincar$\acute{\up{e}}$ type inequality, 
\begin{equation}\label{The-Poinccare-inequality}
\|\P_0 g^\v\|_{L^2_x}=\|a^\v\|_{L^2_x}\lesssim \|\n_x a^\v\|_{L^2_x},
\end{equation}
since 
\begin{equation*}
    \int_{\T^3}a^\v  \, \d x= \int_{\T^3}-\Delta_x\phi^\v \, \d x=0.
\end{equation*}

\textbf{Step 3:} Multiplying $\eqref{eq:scaled-VPFP-2}_1$ by $\M$ and integrating with respect to $v$ over $\R^3$, we have,
\begin{equation}\label{The-formally-limits-1}
\p_t a^\v = - \frac{1}{\v}\up{div}_x\l g^\v, v\M\r_v, 
\end{equation}
and then by taking the limit $\v \rightarrow 0$, the left-hand-side above yields that, considering \eqref{The-def-a} and \eqref{eq:convergence-g-v},
\begin{equation}\label{eq:LHS}
    \text{LHS} = \lim\limits_{\v\to 0} \partial_t a^\v = \lim\limits_{\v\to 0}\partial_t \left( \int_{\R^3}g^\v\M\d v \right) = \partial_t \rho
\end{equation}

For the right-hand-side, by $\eqref{eq:scaled-VPFP-2}_1$ and \eqref{eq:convergence-g-v}, we have,
\begin{equation}\label{The-formally-limits-2}
\begin{aligned}
\text{RHS} =& \lim\limits_{\v\to 0}\frac{1}{\v}\up{div}_x\l g^\v, v\M\r_v\\
=&\lim\limits_{\v\to 0}\frac{1}{\v}\up{div}_x\l(\I-\P_0)g^\v, v\M\r_v\\
=&\lim\limits_{\v\to 0}\frac{1}{\v}\up{div}_v\l(\I-\P_0)g^\v, L(v\M)\r_v\\
=&\lim\limits_{\v\to 0}\up{div}_x\l\frac{1}{\v}L(\I-\P_0)g^\v, v\M\r_v\\[5pt]
=&\lim\limits_{\v\to 0} \up{div}_x \Big\l-\v\p_t g^\v-v\c\n_x g^\v-v\c\n_x\phi^\v\M+\n_v g^\v\c\n_x \phi-\frac{g^\v}{2}v\c\n_x\phi^\v, \, v\M \Big\r_v\\[5pt]
= &\up{div}_x \Big\l-v\c\n_x\big[(\rho-1)\M\big]-v\c\n_x\phi\M+\n_v\big[(\rho-1)\M\big]\c\n_x \phi -\frac{(\rho-1)\M}{2}v\c\n_x\phi, \, v\M \Big\r_v\\[5pt]
=&-\Delta_x\rho+\up{div}_x(\rho\n_x\phi),
\end{aligned}  
\end{equation}
where we use the fact $L(v\M)=v\M$ in the third equality above, and the self-adjoint property of $L$ is applied in the fourth equality.

Finally, by collecting \eqref{eq:LHS}, \eqref{eq:convergence-g-v}, \eqref{The-formally-limits-1} and \eqref{The-formally-limits-2}, we can obtain the limiting DDP system \eqref{The-Drift-Diffusion-Possion-system}.
% \begin{equation*}
% \left\{
% \begin{aligned}
% &\p_t\rho=\Delta_x\rho+\up{div}_x(\rho\n_x\phi),\\[4pt]
% &-\Delta_x\phi=\rho-1.\\[4pt]
% \end{aligned}
% \right.
% \end{equation*}

%%%%%%%%%%%%%%%%%%%%%%%%%%%%%%%%%%%%%%%%%%%%%%%%%%%%%%%%%%%%%%%%%%%%%%%%%%%%%%%
%%%%%%%%%%%%%%%%%%%%%%%%%%%%%%%%%%%%%%%%%%%%%%%%%%%%%%%%%%%%%%%%%%%%%
%%%%%%%%%%%%%%%%%%%%%%%%%%%%%%%%%%%%%%%%%%%%%%%%%%%%%%%%%%%%%%%%%%%%%%%

%\section{The proof of Theorem \ref{Main-Limits}}
\section{Energy estimate}
\label{sec:energy}

The essence of the proof for our main theorems is the global-in-time energy estimate \eqref{Uniform energy estimate}, which is uniform for $0 < \v \leq 1$. Our proof of \eqref{Uniform energy estimate} can be outlined as follows: we first present the local well-posedness of scaled VPFP system in Section \ref{subsec:remainder}, and the specific designs of corresponding energy and dissipation functionals are presented in Section \ref{subsec:energy_dissipation}; as the whole principle of the energy estimate is trying to find sufficient dissipative or decay structures to control the singularity terms, we have to capture such ``good" dissipative structure from both microscopic part and macroscopic part (obtained by the Micro-Macro decomposition) of the reminder system in Section \ref{subsec:remainder_kinetic} and \ref{subsec:remainder_fluid}, respectively; we finally summarize all the estimates and conclude the total energy estimate in a well-designed and closed manner in Section \ref{subsec:proof_total_energy}.

\subsection{Micro-Macro decomposition and local well-posedness of VPFP system}
\label{subsec:remainder}

By the Micro-Macro decomposition as in \cite{DFT10}, we decompose $g$ by its macroscopic part $\P g^\v$ and microscopic part $(\I-\P)g$:
\begin{equation}\label{MM}
g^\v=\P g^\v+(\I-\P)g^\v,
\end{equation}
where the projection $\P\up{:}\,L^2_v\to\up{Span}\{\sqrt{M},v_1\sqrt{M},v_2\sqrt{M},v_3\sqrt{M}\}$ is given by
\begin{equation}\label{def-L-P}
\begin{aligned}
&\P=\P_0\oplus \P_1,\, \quad \,\P_0 g^\v :=a^\v\M,\, \quad \,\P_1 g^\v:=v\c b^\v\M
\end{aligned}
\end{equation}
with
\begin{equation}\label{ab}
    a^\v = \int_{\R^3} g^\v \M \, \d v \quad \text{and} \quad b^\v = \int_{\R^3} g^\v v \M \, \d v.
\end{equation}

Furthermore, the Fokker-Planck operator $L$ in \eqref{def-L} can be decomposed by
\begin{equation}\label{Decop-L}
L g^\v=L(\I-\P)g^\v+\P_1g=L(\I-\P)g^\v+v\c b^\v\M.
\end{equation}
and note that $\I-\P_0$, $\I-\P$ is self-adjoint in $H^d_{x,v}$, i.e.,
\begin{equation}
\begin{aligned}
\l\p_x^\alpha(\I-\P_0)f,\p_x^\alpha g\r_{x,v} = \l\p_x^\alpha f,\p_x^\alpha(\I-\P_0)g\r_{x,v}, \quad
\l\p_x^\alpha(\I-\P)f,\p_x^\alpha g\r_{x,v} = &\l\p_x^\alpha f,\p_x^\alpha(\I-\P)g\r_{x,v},\\
\end{aligned}
\end{equation}
for any $\alpha$ and $f,g\in H^d_{x,v}$, and it is easy to verify that
\begin{equation}\label{I-P}
\begin{aligned}
(\I-\P_0)(\I-\P)=\I-\P,\quad (\I-\P_0)(\I-\P_0)=\I-\P_0,\quad (\I-\P)(\I-\P)=\I-\P.
\end{aligned}
\end{equation}
%for any $f\in H^d_{x,v}$.

According to \cite{CJADRJMA11}, the operator $L$ enjoys the dissipative property, i.e., there exists a constant $C_0>0$ such that
\begin{equation}\label{L-dissipation}
C_0\|(\I-\P)g^\v\|_{\nu}^2+\|b^\v\|^2_{L^2_x}  \leq \l L g^\v, g^\v\r_{x,v}.
\end{equation}

Now we are in a position to present the local well-posedness of the VPFP system \eqref{The-VPFP-system} above.

\begin{proposition}\label{Local-in-time}
For any integer $k \geq 3$, there exists $T^*>0$ independent of $\v$, such that for any $t\in[0,T^*]$ and $\v \in (0,1]$, the VPFP system \eqref{The-VPFP-system} admits a unique solution $(g^\v,\n_x\phi^\v)$ satisfying
\begin{equation*}
\begin{aligned}
 &g^\v(t,x,v) \in L^\infty\left(0,T^*;H^k_{x,v}\right),\quad (\I-\P_0)g^\v(t,x,v) \in L^2\left(0,T^*;\mathcal{H}^k_{x,v}\right),\\[5pt]
 &\n_x\phi^\v(t,x) \in L^\infty\left(0,T^*;H^k_x \right) \cap L^2\left(0,T^*;H^k_x\right)
\end{aligned}   
\end{equation*}
with the energy estimate
\begin{equation}\label{The-total-energy-of-local-in-time-solution}
    \frac{1}{2}\sup_{t\in[0,T^*]}\E_k(t)+\tilde{C}\int_0^{T^*}\D_k(t) \,\d t \,\leq\, C\E_k(0),
\end{equation}
where the constants $C,\,\tilde{C}$ are given in Proposition \ref{total_energy_dissipation}.
\end{proposition}

\begin{proof}
The proof is based on the standard fixed point argument. We refer to \cite{GHMAZP10,Hwang-Jang2013} for more details.
\end{proof}
%%%%%%%%%%%%%%%%%%%%%%%%%%%%%%%%%%%%%%%%%%%%%%%%%%%%%%%%%%%%%%%%%%%%%%%%%%%%%%%%%%%%%%%%%%%%%%

\subsection{Total energy estimate}
\label{subsec:energy_dissipation}

To better state the total energy estimate of the VPFP system, we first introduce the temporal energy and dissipation functionals of different parts:
\begin{itemize}

\item Energy and dissipation functionals: for any integer $k\geq 0$,
\begin{equation}\label{part-energy-functionals}
\begin{aligned}
\E_{k,K,1}(t) := &\| g^\v\|^2_{H^k_x L_v^2}+\|\n_x\phi^\v\|^2_{H^k_x},\\[4pt]
\E_{k,K,2}(t) :=& \sum_{|\a|+|\b|\leq k-1}C_{\a,\b}\|\p_x^\a\p_v^{\b+\b'}(\I-\P)g^\v\|^2_{L^2_{x,v}},\\[4pt]
\E_{k,F}(t):=&\|(a^\v,b^\v,\n_x\phi^\v)\|_{H^{k-1}_x}^2+2\sum_{|\a|=0}^{k-1}\sum_{i,j=1}^3\l\p_x^{\a}(\p_{x_i} b_j^\v+\p_{x_j} b_i^\v),\p_x^\a(\I-\P)g^\v(v_iv_j-1)\M\r_{x,v}\\
&+\v\sum_{|\a|=0}^{k-1}\l\p_x^{\a+\a'}a^\v,\p_x^\a b^\v\r_x,\\[4pt]
\D_{k,K,1}(t):=&\frac{1}{\v^2}\big(\|(\I-\P)g^\v\|_{\mathcal{H}^k_x\mathcal{L}^2_v}^2+\|b^\v\|_{H^k_x}^2\big),\\[4pt]
\D_{k,K,2}(t):=&\frac{1}{\v^2}\sum_{|\a|+|\b|\leq k-1}C_{\a,\b}\|\p_x^\a\p_v^{\b+\b'}(\I-\P)g^\v\|_{\nu}^2,\\[4pt]
\D_{k,F}(t):=&\frac{1}{\v}\|(\n_x b^\v,\up{div}_x b^\v)\|^2_{H^{k-1}_x}+\|\n_x a^\v\|^2_{H^{k-1}_x}+\|\n_x\phi^\v\|_{H^k_x}^2,\\[4pt]
\end{aligned}
\end{equation}
where $C_{\a,\b} > 0$ are constants and $|\a'|=|\b'|=1$.

\item Total energy functional $\E(t)$ and dissipation functional $\D(t)$: for any integer $k\geq 0$,
\begin{equation}\label{energy-dissipation-functional}
\begin{aligned}
\E_k(t) :=& \lambda_1\E_{k,K,1}(t)+\lambda_2\E_{k,K,2}(t)+\lambda_3\E_{k,F}(t),\\[5pt]
\D_k(t) :=& \D_{k,K,1}(t)+\D_{k,K,2}(t)+\D_{k,F}(t),
\end{aligned}
\end{equation}
where $\lambda_i > 0, 1\leq i\leq 3$ are the constants given in \eqref{lambda-constants}.
\end{itemize}

Now we are in a position to present the total energy estimate of the VPFP system \eqref{The-VPFP-system}.

\begin{proposition}\label{total_energy_dissipation}
For any integer $k\geq 3$, let $(g^\v,\n_x\phi^\v)$ be the solution to the VPFP system \eqref{The-VPFP-system}, there exist constants $C,\,\tilde{C} > 0$ independent of $\v$ and $t$ such that, for $t \geq 0$,
\begin{equation}\label{total_energy_estimate}
\begin{aligned}
\frac{1}{2}\frac{\d}{\d t}\E_k(t)+\tilde{C}\D_k(t) \leq C\E^{\frac{1}{2}}_k(t)\D_k(t),
\end{aligned}
\end{equation}
where the energy and dissipation functionals $\E_k(t)$ and $\D_k(t)$ are defined in \eqref{energy-dissipation-functional}.
\end{proposition}

We also introduce another type of energy functional $\mathbb{E}_k(t)$ and dissipation functional $\mathbb{D}_k(t)$:
\begin{equation}\label{Def_total_functional}
\begin{aligned}
\mathbb{E}_k(t) := &\ \|g^\v\|_{H^k_x L^2_v}^2+\|\n_v(\I-\P)g^\v\|^2_{H^{k-1}_{x,v}}+\|(a^\v,b^\v)\|_{H^{k-1}_x}^2,\\[6pt]
\mathbb{D}_k(t) := &\ \frac{1}{\v^2} \left(\|(\I-\P)g^\v\|_{\mathcal{H}^k_{x,v}}^2+\|b^\v\|_{H^k_x}^2 \right) +\frac{1}{\v}\|(\n_x b^\v,\up{div}_x b^\v)\|^2_{H^{k-1}_x} + \|\n_x a^\v\|^2_{H^{k-1}_x}+\|\n_x\phi^\v\|_{H^k_x}^2,
\end{aligned}
\end{equation}
and one can directly verify that the two types of definition are equivalent:
\begin{equation}\label{eq:equivalent}
\mathbb{E}_k(t)\sim \E_k(t),  \quad \mathbb{D}_k(t)\sim\D_k(t).
\end{equation}

%We can also check that the energy functional $\mathbb{E}(t)$ is continuous in $t\in(0, T)$.
% Notice that
% \begin{equation}
% \begin{aligned}
% &\Big|\l\n_x^{\a}(\p_{x_i} b_j+\p_{x_j} b_i),\n_x^\a(\I-\P)g(v_iv_j-1)\M\r_{x,v}\Big|\lesssim\|\n_x^{\a+1}b\|_{L^2_x}\|\n_x^\a(\I-\P)g\|_{L^2_{x,v}}\\
% &\lesssim \|\n_x^{\a+1}\int g v\M dv\|_{L^2_x}\|\n_x^\a g\|_{L^2_{x,v}}\lesssim \|\n_x^{\a}g\|_{H^1_x L^2_v}^2\\
% &\v\l\n_x^{\a+1}a,\n_x^{\a}b\r_x\lesssim \v\|\n_x^{\a+1}\int g\M dv\|_{L^2_x}\|\n_x^{\a}\int gv\M dv\|_{L^2_x}\lesssim \|\n_x^\a g\|_{H^1_x L^2_v}^2
% \end{aligned}
% \end{equation}
% for any $0 \leq |\a| \leq 3,\,i,j=1,\cdots,3$, and $\lambda_1$ defined in \eqref{lambda-constants} is small enough, then
% \begin{equation}
% \begin{aligned}
% \|(a,b)\|_{H^3_x}^2-\frac{1}{2}\|g\|_{H^4_xL^2_v}^2\leq\Big|\E_{mi,F}(t)\Big|\leq \|(a,b)\|_{H^3_x}^2+\frac{1}{2}\|g\|_{H^4_xL^2_v}^2.
% \end{aligned}
% \end{equation}
% Therefore, we have
% \begin{equation}
% \begin{aligned}
% \|(a,b)\|_{H^3_x}^2+\frac{1}{2}\|g\|_{H^4_xL^2_v}^2+\|u\|_{H^4_x}\leq\Big|\E_{mi,F}(t)+\E_{mi,K,1}(t)\Big|\leq \|(a,b)\|_{H^3_x}^2+\frac{3}{2}\|g\|_{H^4_xL^2_v}^2+\|u\|_{H^4_x}^2.
% \end{aligned}
% \end{equation}

Based on Proposition \ref{total_energy_dissipation} and the equivalent relation \eqref{eq:equivalent}, we can also obtain the following energy estimate that is equivalent to \eqref{Def_total_functional}.

\begin{corollary}\label{cor:energy-estimate}
For any integer $k\geq\,3$, let $(g^\v,\n_x\phi^\v)$ be the solution to the VPFP system \eqref{The-VPFP-system}, there exists a constant $C > 0$ independent of $\v$ and $\tau$ such that, for any $\tau\in [0,\infty)$,
\begin{equation}
\frac{1}{2} \mathbb{E}_k(\tau) + \int_0^\tau\mathbb{D}_k(s) \,\d s \leq C\sup_{0 \leq s\leq \tau} \mathbb{E}^{\frac{1}{2}}_k(s)\int_0^\tau\mathbb{D}_k(s) \,\d s + C\,\mathbb{E}_k(0).
\end{equation} 
\end{corollary}

In what follows, we will specifically discuss how to make the energy estimate for different parts, and summarize all the parts to conclude Proposition \ref{total_energy_dissipation} (or equivalently Corollary \ref{cor:energy-estimate}) in Section \ref{subsec:proof_total_energy}.

%%%%%%%%%%%%%%%%%%%%%%%%%%%%%%%%%%%%%%%%%%%%%%%%%%%%%%%%%%%%%%%%%%%%%%%%%%%%

\subsection{Energy estimate for the kinetic part}
\label{subsec:remainder_kinetic}

In this subsection, we make the energy estimate for the kinetic part of the VPFP system \eqref{The-VPFP-system}, which essentially relies on the coercivity property of $L$ to produce dissipation.

\begin{lemma}\label{Estimate-kinetic-Mi-1}
For any integer $k\geq 3$, let $(g^\v,\n_x\phi^\v)$ be the solution to the VPFP system \eqref{The-VPFP-system}, then there exist constants $C_1> 0$ independent of $\v$ and $t$ such that, for $t \geq 0$,
\begin{equation}\label{Step-one}
\begin{aligned}
\frac{1}{2}\frac{\d}{\d t}\E_{k,K,1}(t)+C_1\D_{k,K,1}(t) \lesssim \E^{\frac{1}{2}}_k(t)\D_k(t),
\end{aligned}
\end{equation}
where $\E_{k,K,1}(t)$, $\D_{k,K,1}(t)$, $\E_k(t)$, and $\D_k(t)$ are defined in \eqref{part-energy-functionals} and \eqref{energy-dissipation-functional}, respectively.
\end{lemma}

\begin{proof}
Applying the derivative operator $\p_x^\a$ with $0\leq|\a|\leq k$ to the VPFP system \eqref{The-VPFP-system}, multiplying with $\p_x^\a g^\v$, integrating over $x,v$, and integrating by parts, we have,
\begin{multline}\label{Estimate-nabla-g-0}
\frac{1}{2}\frac{\d}{\d t}\|\p_x^\a g^\v\|^2_{L^2_{x,v}}+\frac{C_0}{\v^2}\|\p_x^\a(\I-\P)g^\v\|^2_{\nu}+\frac{1}{\v^2}\|\p_x^\a b^\v\|^2_{L^2_{x}} \\[5pt]
\leq\underbrace{-\frac{1}{\v}\l v\c\n_x\p_x^{\a}\phi^\v\M, \p_x^\a g^\v\r_{x,v}}_{B_{11}}
\underbrace{-\frac{1}
{\v}\l\p_x^\a(\frac{g^\v}{2}v\c\n_x\phi^\v),\p_x^\a g^\v\r_{x,v}}_{B_{12}}+ \underbrace{\frac{1}{\v}\l\p_x^\a (\n_v g^\v\c\n_x\phi^\v),\p_x^\a g^\v\r_{x,v}}_{B_{13}},
\end{multline}
where \eqref{Decop-L} and \eqref{L-dissipation} are utilized.

For $B_{11}$, by using \eqref{The-formally-limits-1}, $\eqref{The-VPFP-system}_2$ and \eqref{ab}, we have
\begin{equation}\label{B11}
\begin{aligned}
B_{11}=-\frac{1}{\v}\l\p_x^\a\n_x\phi^\v,\p_x^\a\l g^\v,v\M\r_v\r_x =&-\frac{1}{\v}\l\p_x^\a\n_x\phi^\v,\p_x^\a b^\v\r_x \\
=&\frac{1}{\v}\l\p_x^\a\phi^\v,\p_x^\a\up{div}_x b^\v\r_x\\
=&-\l\p_x^\a\phi^\v,\p_x^\a\p_t a^\v\r_x \\
=&\l\p_x^\a\phi^\v,\p_t\p_x^\a\Delta_x \phi^\v\r_x=-\frac{1}{2}\frac{\d}{\d t}\|\p_x^\a\n_x\phi^\v\|_{L^2_x}^2.
\end{aligned}
\end{equation}
%%%%%%%%%%%%%%%%%%%%%%%%%%%%%%%%%%%%%%%%%%%%%%%%%%%%%%%%%%%%%%%%

For $B_{12}$, if $|\a|=0$, considering the decomposition \eqref{MM}, then 
\begin{equation}\label{B12-0}
\begin{aligned}
B_{12}=&-\frac{1}{2\v}\l v\c\n_x\phi^\v,|g^\v|^2\r_{x,v}\\
=&-\frac{1}{2\v}\l v\c\n_x\phi^\v,|(\I-\P)g^\v|^2\r_{x,v}-\frac{1}{\v}\l v\c\n_x\phi^\v,(\I-\P)g^\v\P g^\v\r_{x,v}-\frac{1}{2\v}\l v\c\n_x\phi^\v,|\P g^\v|^2\r_{x,v}\\
=&-\frac{1}{2\v}\l v\c\n_x\phi^\v,|(\I-\P)g^\v|^2\r_{x,v}-\frac{1}{\v}\l v\c\n_x\phi^\v,(\I-\P)g^\v v\c b^\v\M\r_{x,v}-\frac{1}{\v}\l \n_x\phi^\v\c b^\v,a^\v\r_x\\
\lesssim&\frac{1}{\v}\|\n_x\phi^\v\|_{L^\infty_x}\|(\I-\P)g^\v\|_{\nu}^2+\frac{1}{\v}\|\n_x\phi^\v\|_{L^4_x}\|(\I-\P)g^\v\|_{L^2_{x,v}}\|b^\v\|_{L^4_x}+\frac{1}{\v}\|\n_x\phi^\v\|_{L^\infty_x}\|b^\v\|_{L^2_x}\|a^\v\|_{L^2_x}\\
\lesssim&\frac{1}{\v}\|\n_x\phi^\v\|_{H^2_x}\|(\I-\P)g^\v\|_{\nu}^2+\frac{1}{\v}\|\n_x\phi^\v\|_{H^1_x}\|(\I-\P)g^\v\|_{\nu}\|b^\v\|_{H^1_x}+\frac{1}{\v}\|\n_x\phi^\v\|_{H^2_x}\| b^\v\|_{L^2_x}\|\n_x a^\v\|_{L^2_x}\\
\lesssim&\E^{\frac{1}{2}}_k(t)\D_k(t),
\end{aligned}
\end{equation}
where the Poincar$\acute{\up{e}}$ inequality \eqref{The-Poinccare-inequality} is used in the second inequality.

If $|\a|\geq 1$, $B_{12}$ is divided into four parts,
\begin{multline}\label{B12-1}
B_{12}=\underbrace{-\frac{1}{2\v}\l v\c\n_x\phi^\v,|\p_x^\a g^\v|^2\r_{x,v}}_{B_{121}}\underbrace{-\frac{1}{2\v}\l v\c\n_x\p_x^\a\phi^\v,g^\v\p_x^\a g^\v\r_{x,v}}_{B_{122}}\\
\underbrace{-\frac{1}{2\v}\sum_{1\leq|\b|\leq|\a|-1}C_\a^\b\l v\c\n_x\p_x^\b\phi^\v,\p_x^{\a-\b}g^\v\p_x^\a g^\v\r_{x,v}}_{B_{123}}.
\end{multline}
Similar to the estimate in \eqref{B12-0}, we can estimate $B_{121}$ as
\begin{equation}\label{B121}
\begin{aligned}
|B_{121}|\lesssim \E^{\frac{1}{2}}_k(t)\D_k(t).
\end{aligned}
\end{equation}
For $B_{122}$, considering the decomposition \eqref{MM} and equation $\eqref{The-VPFP-system}_2$, we have
\begin{equation*}\label{B122-1}
\begin{aligned}
|B_{122}|=&-\frac{1}{2\v}\l v\c\n_x\p_x^\a\phi^\v,(\I-\P)g^\v\p_x^\a(\I-\P)g^\v\r_{x,v}-\frac{1}{2\v}\l v\c\n_x\p_x^\a\phi^\v,\P g^\v\p_x^\a(\I-\P)g^\v\r_{x,v}\\
&-\frac{1}{2\v}\l v\c\n_x\p_x^\a\phi^\v,(\I-\P) g\p_x^\a \P g^\v\r_{x,v}-\frac{1}{2\v}\l v\c\n_x\p_x^\a\phi^\v,\P g^\v\p_x^\a \P g^\v\r_{x,v}\\[4pt]
=&-\frac{1}{2\v}\l v\c\n_x\p_x^\a\phi^\v,(\I-\P)g^\v\p_x^\a(\I-\P)g^\v\r_{x,v}-\frac{1}{2\v}\l v\c\n_x\p_x^\a\phi^\v,v\c b\M\p_x^\a(\I-\P)g^\v\r_{x,v}\\
&-\frac{1}{2\v}\l v\c\n_x\p_x^\a\phi^\v,(\I-\P) g^\v v\c\p_x^\a b^\v\M \r_{x,v}-\frac{1}{2\v}\l v\c\n_x\p_x^\a\phi^\v,(a^\v v\c\p_x^\a b^\v+v\c b^\v \p_x^\a a^\v)M \r_{x,v}\\[4pt]
\lesssim&\frac{1}{\v}\|\p_x^\a\n_x\phi^\v\|_{L^2_x}\|(\I-\P)g^\v\|_{L^\infty_x L^2_v}\|\p_x^\a(\I-\P)g^\v\|_{L^2_{x,v}}+\frac{1}{\v}\|\p_x^\a\n_x\phi^\v\|_{L^2_x}\|b^\v\|_{L^\infty_x}\|\p_x^\a(\I-\P)g^\v\|_{L^2_{x,v}}\\
&+\frac{1}{\v}\|\p_x^\a\n_x\phi^\v\|_{L^2_x}\|(\I-\P)g^\v\|_{L^\infty_x L^2_v}\|\p_x^\a b^\v\|_{L^2_x}+\frac{1}{\v}\|\p_x^\a\n_x\phi^\v\|_{L^2_x}\| b^\v\|_{L^\infty_x}\|\p_x^\a a\|_{L^2_x}\\
&+\Big|\frac{1}{2\v}\l \p_x^{\a-1} b^\v\c\n_x\p_x^{\a-1}\Delta_x\phi^\v,a^\v \r_x\Big|+\Big|\frac{1}{2\v}\l \p_x^{\a-1} b^\v\c\n_x\p_x^\a\phi^\v,\n_x a^\v \r_x\Big|\\[4pt]
\lesssim&\frac{1}{\v}\|\n_x\phi^\v\|_{H^k_x}\|(\I-\P)g^\v\|_{H^2_x L^2_v}\|(\I-\P)g\|_{H^k_x L^2_v}+\frac{1}{\v}\|\n_x\phi^\v\|_{H^k_x}\|b^\v\|_{H^2_x}\|(\I-\P)g^\v\|_{H^k_x L^2_v}\\
&+\frac{1}{\v}\|\p_x^\a\n_x\phi^\v\|_{L^2_x}\|(\I-\P)g^\v\|_{L^\infty_x L^2_v}\|\p_x^\a b^\v\|_{L^2_x}+\frac{1}{\v}\|\n_x\phi\|_{H^k_x}\| b^\v\|_{H^2_x}\|\n_x a^\v\|_{H^{k-1}_x}\\
&+\Big|\frac{1}{2\v}\l \p_x^\a b^\v\c\n_x\p_x^\a\phi^\v, a^\v \r_x\Big|\\[4pt]
\lesssim&\E^{\frac{1}{2}}_k(t)\D_k(t)+\frac{1}{\v}\|\p_x^{\a-1}b^\v\|_{L^2_x}\|\n_x\p_x^{\a-1}a^\v\|_{L^2_x}\|a^\v\|_{L^\infty_x}+\frac{1}{\v}\|\p_x^{\a-1}b^\v\|_{L^4_x}\|\n_x\p_x^\a \phi^\v\|_{L^2_x}\|\n_x a^\v\|_{L^4_x}\\
\lesssim&\E^{\frac{1}{2}}_k(t)\D_k(t)+\frac{1}{\v}\|b^\v\|_{H^
{k-1}_x}\|\n_x a^\v\|_{H^{k-1}_x}\|a^\v\|_{H^2_x}+\frac{1}{\v}\|b^\v\|_{H^
k_x}\|\n_x \phi^\v\|_{H^k_x}\|\n_x a^\v\|_{H^1_x}\\
\lesssim&\E^{\frac{1}{2}}_k(t)\D_k(t),
\end{aligned}
\end{equation*}
Similar to the estimate for $B_{122}$, we can estimate $B_{123}$ as
\begin{equation}\label{B123}
\begin{aligned}
|B_{123}|\lesssim \E^{\frac{1}{2}}_k(t)\D_k(t).
\end{aligned}
\end{equation}
Hence, collecting the previous estimates of $B_{121}$, $B_{122}$, $B_{123}$, we obtain
\begin{equation}\label{B12}
\begin{aligned}
|B_{12}| \lesssim \E^{\frac{1}{2}}_k(t)\D_k(t).
\end{aligned}
\end{equation}

%%%%%%%%%%%%%%%%%%%%%%%%%%%%%%%%%%%%%%%%%%%%%%%%%%%%%%%%%%%%%%%%%%%

For $B_{13}$, if $|\a|=0$, we have
\begin{equation}\label{B13-1}
\begin{aligned}
B_{13}=\frac{1}{2}\l\n_x\phi^\v,\n_v|g^\v|^2\r_{x,v}=0,
\end{aligned}
\end{equation}
if $|\a|\geq 1$, we can follow the estimate of $B_{12}$ and obtain that 
\begin{equation}\label{B13-2}
\begin{aligned}
\big|B_{13}\big|=&\Big|\frac{1}{\v}\l\n_v g^\v\c\n_x\p_x^\a\phi^\v,\p_x^\a g^\v\r_{x,v}+\frac{1}{\v}\sum_{1\leq|\b|\leq|\a|-1}\binom{\a}{\b}\l\p_x^\b\n_v g^\v\c\n_x\p_x^{\a-\b}\phi^\v,\p_x^\a g^\v\r_{x,v}\Big|\\
\lesssim&\E^{\frac{1}{2}}_k(t)\D_k(t).
% \lesssim&\frac{1}{\v}\big|\l\n_v(\I-\P)g\c\n_x\p_x^\a\phi,\p_x^\a(\I-\P)g\r_{x,v}\big|+\frac{1}{\v}\big|\l\n_v \P g\c\n_x\p_x^\a\phi,\p_x^\a(\I-\P)g\r_{x,v}\big|\\
% &+\frac{1}{\v}\big|\l\n_v(\I-\P)g\c\n_x\p_x^\a\phi,\p_x^\a \P g\r_{x,v}\big|+\frac{1}{\v}\big|\l\n_v \P g\c\n_x\p_x^\a\phi,\p_x^\a \P g\r_{x,v}\big|\\
% &+\frac{1}{\v}\sum_{1\leq|\b|\leq|\a|-1}\Big|\l\p_x^\b\n_v(\I-\P)g\c\n_x\p_x^{\a-\b}\phi,\p_x^\a(\I-\P)g\r_{x,v}\Big|\\
% &+\frac{1}{\v}\sum_{1\leq|\b|\leq|\a|-1}\Big|\l\p_x^\b\n_v \P g\c\n_x\p_x^{\a-\b}\phi,\p_x^\a(\I-\P)g\r_{x,v}\Big|\\
% &+\frac{1}{\v}\sum_{1\leq|\b|\leq|\a|-1}\Big|\l\p_x^\b\n_v \P g\c\n_x\p_x^{\a-\b}\phi,\p_x^\a(\I-\P)g\r_{x,v}\Big|\\
% &+\frac{1}{\v}\sum_{1\leq|\b|\leq|\a|-1}\Big|\l\p_x^\b\n_v \P g\c\n_x\p_x^{\a-\b}\phi,\p_x^\a(\I-\P)g\r_{x,v}\Big|\\
\end{aligned}
\end{equation}
Therefore, by combining \eqref{B13-1} and \eqref{B13-2}, we have
\begin{equation}\label{B13}
\begin{aligned}
\big|B_{13}\big|\lesssim&\E^{\frac{1}{2}}_k(t)\D_k(t).
\end{aligned}
\end{equation}

Finally, substituting the estimates \eqref{B11}, \eqref{B12} and \eqref{B13} into \eqref{Estimate-nabla-g-0}, and then summing up $|\a|$ from 0 to $k$, the proof of the lemma is completed.
\end{proof}

%%%%%%%%%%%%%%%%%%%%%%%%%%%%%%%%%%%%%%%%%%%%%%%%%%%%%%%%%%%%
To complete the total energy estimate, we also need to estimate the mixed partial derivative of $(\I-\P) g^\v$. To this end, applying $(\I-\P)$ to both sides of $\eqref{The-VPFP-system}_1$, it follows that
\begin{multline}\label{The-VPFP-system-I-P-g}
    \p_t(\I-\P)g^\v+\frac{1}{\v}(\I-\P)(v\c\n_x g^\v)+\frac{1}{\v}(\I-\P)(v\c\n_x\phi^\v\M) \\
    +\frac{1}{\v}(\I-\P)(\frac{g^\v}{2}v\c\n_x\phi^\v-\n_v g^\v\c\n_x\phi^\v)+\frac{1}{\v^2}L(\I-\P)g^\v=0,
\end{multline}
%where the fact $(\I-\P)L g^\v=L(\I-\P)g^\v$ is used.

In the following Lemma \ref{Estiamte-Mi-kinetic-part-mixed-partial}, we present the energy estimate of $(\I-\P) g^\v$.
\begin{lemma}\label{Estiamte-Mi-kinetic-part-mixed-partial}
For any integer $k\geq3$, let $(g^\v,\n_x\phi^\v)$ be the solution to the VPFP system \eqref{The-VPFP-system}, then there exist constants $C_2,\,\tilde{C}_2 > 0$ independent of $\v$ and $t$ such that, for $t \geq 0$,
\begin{equation}\label{Step-two}
\frac{1}{2}\frac{\d}{\d t}\E_{k,K,2}(t)+C_2\D_{k,K,2}(t)-\tilde{C
}_2\big(\D_{k,F}(t)+\D_{k,K,1}(t)\big)\lesssim \E^{\frac{1}{2}}_k(t)\D_k(t),
\end{equation}
where the energy and dissipation functionals $\E_{k,K,2}(t)$, $\D_{k,K,1},\,\D_{k,K,2}(t),\,\D_{k,F}(t)$, $\E_k(t)$, and $\D_k(t)$ are defined in \eqref{part-energy-functionals} and \eqref{energy-dissipation-functional}, respectively.
\end{lemma}

\begin{proof}
Applying $\p_x^\a \p_v^\b$ with $1\leq|\a|+|\b|\leq k$ and $ |\b| \geq 1 $ to \eqref{The-VPFP-system-I-P-g}, multiplying by $\p_x^\a\p_v^\b(\I-\P) g^\v$, and integrating over $x,v$, we obtain, for $|\b'|=1$ and $|\b''|=2$,
\begin{equation}\label{Estimate-mix-partial-I-P-g-1}
\begin{aligned}
& \ \frac{1}{2}\frac{\d}{\d t}\|\p_x^\a\p_v^\b(\I-\P)g^\v\|^2_{L^2_{x,v}} + \frac{C_0}{\v^2}\|(\I-\P_0)\p_x^\a\p_v^\b(\I-\P)g^\v\|^2_{\nu} \\[5pt]
\leq &\underbrace{-\frac{ \binom{\b}{\b'} }{2\v^2}\l v\p_x^\a\p_v^{\b-\b'}(\I-\P)g^\v,\p_x^\a\p_v^\b(\I-\P)g^\v\r_{x,v}}_{B_{21}} \underbrace{-\frac{\binom{\b}{\b''}}{2\v^2}\l \p_x^\a\p_v^{\b-\b''}(\I-\P)g^\v,\p_x^\a\p_v^{\b-\b''}\Delta_v(\I-\P)g^\v\r_{x,v}}_{B_{22}} \\
& \underbrace{-\frac{1}{\v}\l\p_x^\a\p_v^\b(\I-\P)(v\c\n_x g^\v),\p_x^\a\p_v^\b(\I-\P)g^\v\r_{x,v}}_{B_{23}} \underbrace{-\frac{1}{\v}\l\p_x^\a\p_v^\b(\I-\P)(\frac{g^\v}{2}v\c\n_x\phi^\v),\p_x^\a\p_v^\b(\I-\P)g^\v\r_{x,v}}_{B_{24}}\\
& +\underbrace{\frac{1}{\v}\l\p_x^\a\p_v^\b(\I-\P)(\n_v g^\v\c\n_x\phi^\v),\p_x^\a\p_v^\b(\I-\P)g^\v\r_{x,v}}_{B_{25}},
\end{aligned}
\end{equation}
where we use the coercivity estimate \eqref{L-dissipation} and the following direct calculations:
\begin{equation*}
    (\I-\P)(v\c\n_x\phi^\v\M)=0.
\end{equation*}
% and
% \begin{equation*}
%     \p_x^\a\p_v^\b L[(\I-\P)g^\v]= L[\p_x^\a\p_v^\b(\I-\P)g^\v] + \frac{\binom{\b}{\b'}}{2} v  \p_x^\a\p_v^{\b-\b'}(\I-\P)g^\v + \frac{\binom{\b}{\b''}}{2}\p_x^\a\p_v^{\b-\b''}(\I-\P)g^\v.
% \end{equation*}

Notice that, for $\P_0$ in \eqref{def-L-P} and any $h$,
\begin{equation}\label{Notice}
    \begin{aligned}
        \|(\I-\P_0)\p_v^\b h\|_{\nu}^2=& \Big\l\n_v\big[(\I-\P_0)\p_v^\b h\big],\n_v\big[(\I-\P_0)\p_v^\b h\big]\Big\r_{x,v} + \Big\l 1+|v|^2,|(\I-\P_0)\p_v^\b h|^2\Big\r_{x,v}\\[5pt]
        =&\|\n_v\p_v^\b h\|_{L^2_{x,v}}^2-2 \Big\l\n_v\P_0\p_v^\b h,\n_v\p_v^\b h\Big\r_{x,v} + \|\n_v\P_0\p_v^\b h\|_{L^2_{x,v}}^2\\[5pt]
        & + \|\sqrt{1+|v|^2}\p_v^\b h\|_{L^2_{x,v}}^2-2\Big\l(1+|v|^2)\P_0\p_v^\b h,\p_v^\b h \Big\r_{x,v} + \|\sqrt{1+|v|^2}\P_0\p_v^\b h\|_{L^2_{x,v}}^2\\[5pt]
        \geq&\ \|\n_v\p_v^\b h\|_{L^2_{x,v}}^2 - 2\Big\l\n_v\big[\l h,P_{\b}\M\r_v\M\big],\n_v\p_v^\b h \Big\r_{x,v}\\[5pt]
        &+\|\sqrt{1+|v|^2}\p_v^\b h\|_{L^2_{x,v}}^2-2\Big\l(1+|v|^2)\big(\l h,P_{\b}\M\r_v\M\big),\p_v^\b h\Big\r_{x,v}\\[5pt]
        \geq&\ \frac{1}{2}\big(\|\n_v\p_v^\b h\|_{L^2_{x,v}}^2+\|\sqrt{1+|v|^2}\p_v^\b h\|_{L^2_{x,v}}^2\big)-C\|h\|_{L^2_{x,v}}^2\\[5pt]
        =&\|\p_v^\b h\|_{\nu}^2-C\|h\|_{L^2_{x,v}}^2,
    \end{aligned}
\end{equation}
where $\int_{\R^3}\p_v^\b h\M\d v=\int_{\R^3}h P_{\b}(v)\M \d v$ with polynomial function $P_{\b}(v)$.
Then, choosing $h=\p_x^\a(\I-\P)g^\v$ in \eqref{Notice} above, we have
\begin{equation}\label{B20}
\begin{aligned}
\frac{C_0}{\v^2}\|(\I-\P_0)\p_x^\a\p_v^\b(\I-\P)g^\v\|^2_{\nu} \,\geq\, \frac{C_0}{2\v^2}\|\p_x^\a\p_v^\b(\I-\P)g^\v\|^2_{\nu}-\frac{C}{\v^2}\|\p_x^\a(\I-\P)g^\v\|_{\nu}^2.
\end{aligned}  
\end{equation}

For $B_{21}$ and $B_{22}$, by applying integration by parts, for $|\b'|=1$ and $|\b''|=2$,
\begin{equation}\label{B21-B22}
\begin{aligned}
|B_{21}| + |B_{22}|= &\ \Big|-\frac{\binom{\b}{\b'}}{2\v^2} \big\l v\p_x^\a\p_v^{\b-\b'}(\I-\P)g^\v,\p_x^\a\p_v^\b(\I-\P)g^\v \big\r_{x,v} \Big| \\
& \qquad \qquad \qquad \qquad \qquad + \Big| -\frac{\binom{\b}{\b''}}{2\v^2} \big\l\p_x^\a\p_v^{\b-\b''}(\I-\P)g^\v,\p_x^\a\p_v^{\b-\b''}\Delta_v(\I-\P)g^\v \big\r_{x,v} \Big|\\[5pt]
\leq&\ \frac{C_0}{2^6\v^2}\|\p_x^\a\p_v^\b(\I-\P)g^\v\|_{\nu}^2+\frac{C}{\v^2}\|\p_x^\a\p_v^{\b-\b'}(\I-\P)g^\v\|^2_{L^2_{x,v}},
\end{aligned}
\end{equation}
where the Young inequality and H$\ddot{\up{o}}$lder inequality are used in the inequality.

For $B_{23}$, it can be divided into three parts:
\begin{equation}\label{B23-1}
\begin{aligned}
B_{23}=&-\frac{1}{\v}\l\p_x^\a\p_v^\b(\I-\P)(v\c\n_x g^\v),\p_x^\a\p_v^\b(\I-\P)g^\v\r_{x,v}\\[5pt]
=&\underbrace{-\frac{1}{\v}\l\p_x^\a\p_v^\b\big(v\c\n_x(\I-\P)g^\v\big),\p_x^\a\p_v^\b(\I-\P)g^\v\r_{x,v}}_{B_{231}}\underbrace{-\frac{1}{\v}\l\p_x^\a\p_v^\b(v\c\n_x \P g^\v),\p_x^\a\p_v^\b(\I-\P)g^\v\r_{x,v}}_{B_{232}}\\
&+\underbrace{\frac{1}{\v}\l\p_x^\a\p_v^\b\P\big(v\c\n_x g^\v\big),\p_x^\a\p_v^\b(\I-\P)g^\v\r_{x,v}}_{B_{233}}.
\end{aligned}
\end{equation}
For $B_{231}$, we find, for $|\a'|=|\b'|=1$,
\begin{equation*}\label{B431}
\begin{aligned}
|B_{231}|=&\Big|-\frac{1}{2\v} \Big\l \n_x|\p_x^\a\p_v^\b(\I-\P)g^\v|^2,v\Big\r_{x,v}-\frac{\binom{\b}{\b-\b'}}{\v} \Big\l\p_x^{\a+\a'}\p_v^{\b-\b'}(\I-\P)g^\v,\p_x^\a\p_v^\b(\I-\P)g^\v \Big\r_{x,v}\Big|  \\[5pt]
\leq&\ \frac{C_0}{2^8\v^2}\|\p_x^\a\p_v^\b(\I-\P)g^\v\|^2_{L^2_{x,v}} + C\|\p_x^{\a+\a'}\p_v^{\b-\b'}(\I-\P)g^\v\|^2_{L^2_{x,v}}\\[5pt]
\leq&\ \frac{C_0}{2^8\v^2}\|\p_x^\a\p_v^\b(\I-\P)g^\v\|^2_{L^2_{x,v}} + \frac{C}{\v^2}\|\p_x^{\a+\a'}\p_v^{\b-\b'}(\I-\P)g^\v\|^2_{L^2_{x,v}}.
\end{aligned}
\end{equation*}
For $B_{232}$,  we have, for $|\a'|=|\b'|=1$,
\begin{equation*}
\begin{aligned}
|B_{232}|=&\ \Big| -\frac{1}{\v} \Big\l\p_x^\a\p_v^\b[(v\c\n_x a^\v+v\otimes v:\n_x b^\v)\M],\p_x^\a\p_v^\b(\I-\P)g^\v \Big\r_{x,v} \Big|\\[4pt]
\leq&\ \frac{C_0}{2^8\v^2} \|\p_x^\a\p_v^\b(\I-\P)g^\v\|^2_{L^2_{x,v}} + C\|(\p_x^{\a+\a'}a^\v,\p_x^{\a+\a'}b^\v)\|^2_{L^2_{x}}\\[4pt]
\leq&\ \frac{C_0}{2^8\v^2} \|\p_x^\a\p_v^\b(\I-\P)g^\v\|^2_{L^2_{x,v}} + C\D_{k,F}(t).
\end{aligned}
\end{equation*}
For $B_{233}$, by further noticing that
\begin{equation*}
\begin{aligned}
\P(v\c\n_x g^\v)=&\ \P[v\c\n_x\P g^\v]+\P[v\c\n_x(\I-\P)g^\v]\\[4pt]
=& \ (v\c\n_x a^\v+\up{div}_x b^\v)\M+\P[v\c\n_x(\I-\P)g^\v]\\[4pt]
=&\ \big[v\c\n_x a^\v+\up{div}_x b^\v+\l v\c\n_x(\I-\P)g^\v, \M\r_v+v\c\l v\c\n_x(\I-\P)g^\v,v\M\r_v\big]\M,
\end{aligned}
\end{equation*}
we have, for $|\a'|=1$,
\begin{equation*}
\begin{aligned}
|B_{233}|=&\ \Big|\frac{1}{\v} \l \p_x^\a\p_v^\b \P (v\c\n_x g^\v), \p_x^\a\p_v^\b(\I-\P)g^\v \r_{x,v} \Big|\\[4pt]
\leq &\ \frac{C_0}{2^8\v^2} \|\p_x^\a\p_v^\b(\I-\P)g^\v\|^2_{L^2_{x,v}}  +C\|(\p_x^{\a+\a'}a^\v,\p_x^\a\up{div}_x b^\v)\|^2_{L^2_{x}} +C\|\p_x^{\a+\a'}(\I-\P)g^\v\|^2_{L^2_{x,v}} \\[4pt]
\leq &\ \frac{C_0}{2^8\v^2}\|\p_x^\a\p_v^\b(\I-\P)g^\v\|^2_{L^2_{x,v}} + C\big(\D_{k,F}(t)+\D_{k,K,1}(t)\big).
\end{aligned}
\end{equation*}
Therefore, we have, for $|\a'|=|\b'|=1$,
\begin{equation}\label{B23}
\begin{aligned}
|B_{23}|\leq &\ \frac{C_0}{2^6\v^2} \|\p_x^\a\p_v^\b(\I-\P)g^\v\|^2_{L^2_{x,v}} + \frac{C}{\v^2}\|\p_x^{\a+\a'}\p_v^{\b-\b'}(\I-\P)g^\v\|^2_{L^2_{x,v}}+C\Big(\D_{k,F}(t)+\D_{k,K,1}(t)\Big),
\end{aligned}
\end{equation}
where $\D_{mi,K,1}(t)$, $\D_{mi,F}(t)$ are defined in \eqref{part-energy-functionals}.

By using the Micro-Macro decomposition \eqref{MM} and the similar estimate for $B_{122}$, $B_{24}$ can be bounded by
\begin{equation}\label{B24}
\begin{aligned}
|B_{24}|\leq \, &\frac{1}{2\v}\Big| \Big\l\p_x^\a\p_v^\b\Big\{\big[(\I-\P)g^\v \big]v\c\n_x\phi^\v\Big\},\p_x^\a\p_v^\b(\I-\P)g^\v \Big\r_{x,v}\Big|\\[4pt]
&+\frac{1}{2\v}\Big| \Big\l\p_x^\a\p_v^\b\big[(\P g^\v) v\c\n_x\phi^\v\big], \p_x^\a\p_v^\b(\I-\P)g^\v \Big\r_{x,v}\Big|\\[4pt]
&+\frac{1}{2\v}\Big|\Big\l\p_x^\a\p_v^\b\Big\{\big[(\I-\P) g^\v\big] v\c\n_x\phi^\v\Big\},\p_x^\a\p_v^\b(\I-\P)g^\v \Big\r_{x,v}\Big|\\[4pt]
&+\frac{1}{2\v}\Big| \Big\l\p_x^\a\p_v^\b\big[(\P g^\v) v\c\n_x\phi^\v\big],\p_x^\a\p_v^\b(\I-\P)g^\v \Big\r_{x,v}\Big|\\[4pt]
\lesssim \, &\E^{\frac{1}{2}}_k(t)\D_k(t).
\end{aligned}
\end{equation}
In addition, the similar estimate of $B_{13}$ yields
\begin{equation}\label{B25}
\begin{aligned}
|B_{25}|\lesssim \E^{\frac{1}{2}}_k(t)\D_k(t).
\end{aligned}
\end{equation}

Note that there is a term $\frac{1}{\v^2}\|\p_x^\a\p_v^{\b-\b'}(\I-\P)g^\v\|^2_{L^2_{x,v}}$ in \eqref{B21-B22} and \eqref{B23}, which is still not well-controlled. However, observing that the orders of $v$-derivatives in this term is $|\b|-1$, we can employ an induction over $|\b|$ and then collect the estimates \eqref{Estimate-mix-partial-I-P-g-1}, \eqref{B21-B22}, \eqref{B23}, \eqref{B24}, \eqref{B25} to find that there exist constants $C_2,\,\tilde{C}_2 >0 $ such that
\begin{equation*}
\begin{aligned}
\frac{1}{2}\frac{\d}{\d t}\E_{k,K,2}(t) + C_2\D_{k,K,2}(t)-\tilde{C}_2\big( \D_{k,K,1}(t)+\D_{k,F}(t)\big)\lesssim  \E^{\frac{1}{2}}_k(t)\D_k(t).
\end{aligned}
\end{equation*}
This completes the proof of the lemma.
\end{proof}

%%%%%%%%%%%%%%%%%%%%%%%%%%%%%%%%%%%%%%%%%%%%%%%%%%%%%%%%%%%%%%%%%%%%%%%

\subsection{Energy estimate for the macroscopic part}
\label{subsec:remainder_fluid}

In this subsection, we need to show the energy dissipation rate of the macroscopic parts $\frac{1}{\v} \| \n_x b^\v\|^2_{L^2_x}+\|\n_x a^\v\|^2_{L^2_x}$. Inspired by \cite{CJADRJMA11}, we propose the following auxiliary hyperbolic-parabolic coupled system of $a^\v$ and $b^\v$: for $1\leq i,j\leq 3$,
\begin{equation}\label{Pg-a-b}
\left\{
\begin{aligned}
&\p_t a^\v+\frac{1}{\v}\up{div}_x b^\v=0,\\
&\p_t b_i^\v+\frac{1}{\v}\p_{x_i} a^\v +\frac{1}{\v}\p_{x_i}\phi^\v+\frac{1}{\v^2} b_i^\v+\frac{1}{\v}a^\v\p_{x_i}\phi^\v+\frac{1}{\v}\sum_{k=1}^3\p_{x_k}\Gamma_{ik}[(\I-\P)g^\v]= 0,\\
&\frac{1}{\v}(\p_{x_i}b_j^\v+\p_{x_j}b_i^\v)+\frac{1}{\v}(b_j^\v\p_{x_i}\phi^\v+b_i^\v\p_{x_j}\phi^\v) =- \p_t \Gamma_{ij}[(\I-\P)g^\v]-\frac{2}{\v^2}\Gamma_{ij}[(\I-\P)g^\v]-\frac{1}{\v}\Gamma_{ij}\big[v\c\n_x(\I-\P)g^\v\big],\\
\end{aligned}
\right.
\end{equation}
where $a^\v $ and $b^\v$ are defined as in \eqref{ab}, and $\Gamma_{ij}$ are given by
\begin{equation}\label{R4i-R5i-Gamma-l}
\begin{aligned}
\Gamma_{ij}[g^\v]=&\ \int g^\v(v_i v_j-1)\M \,\d v.
\end{aligned}
\end{equation}

\begin{lemma}\label{Estiamte-Mi-fluid-part}
For any integer $k\geq3$, let $(g^\v,\n_x\phi^\v)$ be the solution to the VPFP system \eqref{The-VPFP-system}, there exist constants $C_3,\,\tilde{C}_3 > 0$ independent of $\v$ and $t$ such that, for any $t \geq 0$,
\begin{equation}\label{last-step}
\begin{aligned}
\frac{1}{2}\frac{\d}{\d t}\E_{k,F}(t)+C_3\D_{k,F}(t)-\tilde{C}_3 \D_{k,K,1}(t)\lesssim \E^{\frac{1}{2}}_k(t)\D_k(t),
\end{aligned}
\end{equation}
where the energy and dissipation functionals $\E_{k,F}(t)$, $\D_{k,F}(t)$, $\D_{k,K,1}(t)$, $\E_k(t)$ and $\D_k(t)$ are defined in \eqref{part-energy-functionals}, and \eqref{energy-dissipation-functional}, respectively.
\end{lemma}

\begin{proof}
By applying the derivative operator $\p_x^\a$ with $0\leq|\a|\leq k-1$ to $\eqref{Pg-a-b}_2$, multiplying with $\p_x^\a b^\v$, and then integrating over $x$, we have, for $|\a'|=1$,
\begin{multline}\label{Energy of b}
\frac{1}{2}\frac{\d}{\d t}\|\p_x^\a b^\v\|^2_{L^2_x} +\frac{1}{\v^2}\|\p_x^\a b^\v\|_{L^2_x}^2+ \underbrace{\frac{1}{\v}\l\p_x^{\a+\a'} a^\v,\p_x^\a b^\v\r_x}_{B_{31}} +\underbrace{\frac{1}{\v}\l\p_x^{\a+\a'}\phi,\p_x^\a b^\v\r_x}_{B_{32}}\\
+\underbrace{\frac{1}{\v}\sum_{i=1}^3\l\p_x^\a(a^\v\p_{x_i}\phi^\v),\p_x^\a b_i^\v\r_v}_{B_{33}}
+\underbrace{\frac{1}{\v}\sum_{i,j=1}^3\l\p_x^{\a}\Gamma_{ij}[\p_{x_j}(\I-\P)g^\v],\p_x^\a b_i^\v\r_x}_{B_{34}}=0.
\end{multline}

For $B_{31}$, by substituting $\eqref{Pg-a-b}_1$, we have, for $|\a'|=1$,
\begin{equation}\label{B31}
\begin{aligned}
B_{31}= \frac{1}{\v}\l\p_x^{\a+\a'}a^\v,\p_x^\a b^\v\r_x = -\frac{1}{\v}\l\p_x^\a a^\v,\p_x^\a\up{div}_xb\r_x=\l\p_x^\a a^\v,\p_x^\a\p_t a^\v\r_x=\frac{1}{2}\frac{\d}{\d t}\|\p_x^\a a^\v\|^2_{L^2_x}.
\end{aligned}
\end{equation}

For $B_{32}$, by substituting $\eqref{Pg-a-b}_1$ and $\eqref{The-VPFP-system}_2$, we find, for $|\a'|=1$,
\begin{equation}\label{B32}
\begin{aligned}
B_{32}=\frac{1}{\v}\l\p_x^{\a+\a'}\phi^\v,\p_x^\a b^\v\r_x = -\frac{1}{\v}\l\p_x^\a \phi^\v,\p_x^\a\up{div}_x b^\v\r_x
=\l\p_x^\a a^\v,\p_x^\a\p_t a^\v\r_x
=&-\l\p_x^\a\phi^\v,\p_t\p_x^\a\Delta_x\phi^\v\r_x\\
=&\frac{1}{2}\frac{\d}{\d t}\|\p_x^{\a+\a'}\phi^\v\||_{L^2_x}^2.
\end{aligned}
\end{equation}

For $B_{33}$, if $|\a|=0$, we have
\begin{equation}\label{B33-0}
\begin{aligned}
|B_{33}|=\frac{1}{\v}\big|\l a^\v,b^\v\c\n_x\phi^\v\r_x\big|
\leq&\frac{1}{\v}\|\n_x\phi^\v\|_{L^\infty_x}\|a^\v\|_{L^2_x}\|b^\v\|_{L^2_x}\\
\lesssim& \frac{1}{\v}\|\n_x\phi^\v\|_{H^2_x}\|\n_x a^\v\|_{L^2_x}\|b^\v\|_{L^2_x}\\
\lesssim&\E^{\frac{1}{2}}_k\D_k(t)),
\end{aligned}
\end{equation}
if $|\a|\geq 1$, we have
\begin{equation}\label{B33-1}
\begin{aligned}
|B_{33}|\lesssim&\frac{1}{\v}\big|\l\p_x^\a a^\v,\n_x\phi^\v\c\p_x^\a b^\v\r_x\big|+\frac{1}{\v}\sum_{0\leq|\a'|\leq|\a|-1}\big|\l\p_x^{\a'}a^\v,\p_x^{\a-\a'}\phi^\v\c\p_x^\a b^\v\r_x\big|\\
\lesssim&\frac{1}{\v}\|\p_x^\a a^\v\|_{L^2_x}\|\n_x\phi^\v\|_{L^\infty_x}\|\p_x^\a b^\v\|_{L^2_x}+\frac{1}{\v}\sum_{0\leq|\a'|\leq|\a|-1}\|\p_x^{\a'}a^\v\|_{L^2_x}\|\p_x^{\a-\a'}\phi^\v\|_{L^4_x}\|\p_x^\a b^\v\|_{L^4_x}\\
\lesssim&\frac{1}{\v}\|\n_x\phi^\v\|_{H^2_x}\|\n_x a^\v\|_{H^{k-1}_x}\|b^\v\|_{H^k_x}+\frac{1}{\v}\|\n_x\phi^\v\|_{H^k_x}\|\n_x a^\v\|_{H^{k-1}_x}\| b^\v\|_{H^k_x}\\
\lesssim&\E^{\frac{1}{2}}_k(t)\D_k(t),
\end{aligned}
\end{equation}
where the H$\ddot{\up{o}}$lder inequality and the Poincar$\acute{\up{e}}$ inequality in \eqref{The-Poinccare-inequality} are used.

For $B_{34}$, considering the smallness of $\v$, we obtain
\begin{equation}\label{B34}
    \begin{aligned}
        |B_{34}|\leq &\, \frac{1}{\v}\sum_{i,j=1}\Big|\l\p_x^\a\Gamma_{ij}\big[\p_{x_j}(\I-\P)g^\v\big],\p_x^\a b_i^\v\r_x\Big|\\
        \leq&\, \frac{1}{\v}\sum_{i,j=1}\|\p_x^\a\Gamma_{ij}\big[\p_{x_j}(\I-\P)g^\v\big]\|_{L^2_{x,v}}\|\p_x^\a b_i^\v\|_{L^2_x}\\
        \leq&\, \frac{2^8}{\v}\|(\I-\P)g^\v\|_{\mathcal{H}^k_x\mathcal{L}^2_v} \|\p_x^\a b^\v\|_{L^2_x}\\
        \leq&\, \frac{2^{16}}{\v^2}\|(\I-\P)g^\v\|_{\mathcal{H}^k_x\mathcal{L}^2_v}^2+\frac{1}{\v^2}\|\p_x^\a b^\v\|_{L^2_x}^2\\
        \leq&\, 2^{16}\D_{k,K,1}(t)+\frac{1}{\v^2}\|\p_x^\a b^\v\|_{L^2_x}^2.
    \end{aligned}
\end{equation}
%\kq{The definition of $\mathcal{H}^k_x\mathcal{L}^2_v$ see $\mathcal{H}^d_x \mathcal{H}^e_v:=\big\{f(x,v) \ \big| \ \|\p_x^\alpha\p_v^\beta f\|_{\nu} <\infty,\,\up{for any}\,|\a|\leq d,\,|\b|\leq e\big\}$ when $d=k,e=0$. Put it in notation part}
%for some positive constant $\tilde{C}_3>0$. \kq{what is $C_3$}

Therefore, by substituting the estimates of $B_{31}$, $B_{32}$, $B_{33}$, and $B_{34}$ in \eqref{B31}-\eqref{B34} into \eqref{Energy of b}, we find,
\begin{equation}\label{Energy-of-a-and-b}
\begin{aligned}
\frac{1}{2}\frac{\d}{\d t}\|\big(\p_x^\a a^\v,\p_x^\a b^\v,\p_x^\a\n_x\phi^\v\big)\|^2_{L^2_x}-2^{16}\D_{k,K,1}(t)\lesssim \E^{\frac{1}{2}}_k(t)\D_k(t).
\end{aligned}
\end{equation}

%%%%%%%%%%%%%%%%%%%%%%%%%%%%%%%%%%%%%%%%%%
On the other hand, it follows from $\eqref{Pg-a-b}_3$ that
\begin{multline}\label{partial-b-1}
\frac{1}{\v}\sum_{i,j=1}^3\|\p_x^\a(\p_{x_i} b_j^\v+\p_{x_j}b_i^\v)\|^2_{L^2_x} = -\frac{\d}{\d t}\sum_{i,j=1}^3\int\p_x^\a(\p_{x_i} b_j^\v+\p_{x_j}b_i^\v)\p_x^\a\Gamma_{ij}[(\I-\P)g^\v] \,\d x\\
 +\underbrace{\sum_{i,j=1}^3\l\p_x^\a (\p_{x_j}\p_t b_i^\v+\p_{x_i}\p_t b_j^\v),\p_x^\a\Gamma_{ij}[(\I-\P)g^\v]\r_x}_{B_{41}}\underbrace{-\frac{1}{\v}\sum_{i,j=1}^3\l\p_x^\a(\p_{x_i}b_j^\v+\p_{x_j}b_i^\v),\p_x^\a(b_i^\v\p_{x_j}\phi^\v+b_j\p_{x_i}\phi^\v)\r_x}_{B_{42}}\\
\underbrace{-\sum_{i,j=1}^3\frac{2}{\v^2}\l\p_x^\a(\p_{x_j}b_i^\v+\p_{x_i}b_j^\v),\p_x^\a\Gamma_{ij}[(\I-\P)g^\v]\r_x}_{B_{43}}\underbrace{-\frac{1}{\v}\sum_{i,j=1}^3\l\p_x^\a(\p_{x_j}b_i^\v+\p_{x_i}b_j^\v),\p_x^\a\Gamma_{ij}\big[v\c\n_x(\I-\P)g^\v\big]\r_x}_{B_{44}}
\end{multline}
for $0 \leq |\a| \leq k-1$.

For $B_{41}$, replacing $\p_t b_i^\v$ by $\eqref{Pg-a-b}_2$, it can be further divided into the following parts:
\begin{equation}\label{B51-0}
\begin{aligned}
B_{41}=&\ \underbrace{\frac{2}{\v}\sum_{i,j=1}^3 \l\p_x^\a \p_{x_i} a^\v,\p_x^\a\p_{x_j} \Gamma_{ij}[(\I-\P)g^\v]\r_x}_{B_{411}} +\underbrace{\frac{2}{\v}\sum_{i,j=1}^3\l\p_x^\a\p_{x_i}\phi^\v,\p_x^\a\p_{x_j}\Gamma_{ij}[(\I-\P)g^\v]\r_x}_{B_{412}}\\
&+\underbrace{\frac{2}{\v}\sum_{i,j=1}^3\l\p_x^\a b_i^\v,\p_x^\a\p_{x_j}\Gamma_{ij}[(\I-\P)g^\v]\r_x}_{B_{413}}+\underbrace{\frac{2}{\v^2}\sum_{i,j=1}^3\l\p_x^\a(a^\v\p_{x_i}\phi^\v),\p_x^\a\p_{x_j}\Gamma_{ij}[(\I-\P)g^\v]\r_x}_{B_{414}}\\
&\ +\underbrace{\frac{2}{\v}\sum_{i,j,k=1}^3\l\p_x^\a \p_{x_k}(\I-\P)g^\v(v_k v_i-1),\p_x^\a\p_{x_j}\Gamma_{ij}[(\I-\P)g^\v]\M\r_{x,v}}_{B_{415}},
\end{aligned}
\end{equation}
where, in the last term above, we use the fact that $\Gamma_{ij}[\p_{x_j}(\I-\P)g^\v]=\l \p_{x_j}(\I-\P)g^\v,\M \r_v$.\\
For $B_{411}$, we have
\begin{equation}\label{B411}
\begin{aligned}
|B_{411}|\leq & \, \frac{2}{\v}\sum_{i,j=1}^3\|\p_x^\a \p_{x_i} a^\v\|_{L^2_x}\|\p_x^\a\p_{x_j} \Gamma_{ij}[(\I-\P)g^\v]\|_{L^2_x} \\
\leq&\, \frac{2^{16}}{\v^2}\|(\I-\P)g^\v\|_{\mathcal{H}^k_x\mathcal{L}^2_v}^2+\frac{1}{2^8}\|\p_x^{\a}\n_x a^\v\|_{L^2_x}^2\\
\leq& \, 2^{16}\D_{k,K,1}(t)+\frac{1}{2^8}\|\p_x^{\a}\n_x a^\v\|_{L^2_x}^2,
\end{aligned}
\end{equation}
where the H$\ddot{\up{o}}$lder inequality in $v$ and the Young inequality are applied.\\
For $B_{412}$, we have
%using the Possion equation $-\Delta_x\phi=a$, we have, for $|\a'|=1$, 
\begin{equation}\label{B412}
\begin{aligned}
|B_{412}|\leq& \, \frac{2}{\v}\sum_{i,j=1}^3\|\p_x^\a \p_{x_i} \phi^\v\|_{L^2_x}\|\p_x^\a\p_{x_j} \Gamma_{ij}[(\I-\P)g^\v]\|_{L^2_x}\\
\leq& \, \frac{2^{16}}{\v^2}\|(\I-\P)g^\v\|_{\mathcal{H}^k_x\mathcal{L}^2_v}^2+\frac{1}{2^8}\|\p_x^{\a}\n_x\phi^\v\|_{L^2_x}^2\\
\leq& \, 2^{16}\D_{k,K,1}(t)+\frac{1}{2^8}\|\p_x^{\a}\n_x \phi^\v\|_{L^2_x}^2.
\end{aligned}
\end{equation}
For $B_{413}$, we have 
\begin{equation}\label{B413}
\begin{aligned}
|B_{413}|\leq& \, \frac{2}{\v}\sum_{i,j=1}^3\|\p_x^\a \p_{x_i}b_j^\v\|_{L^2_x}\|\p_x^\a\p_{x_j} \Gamma_{ij}[(\I-\P)g^\v]\|_{L^2_x}\\
\leq& \, \frac{2^{16}}{\v^2}\|(\I-\P)g^\v\|_{\mathcal{H}^k_x\mathcal{L}^2_v}^2+\frac{2^{16}}{\v^2}\|\p_x^{\a}\n_x b^\v\|_{L^2_x}^2\\
\leq& \, 2^{16}\D_{k,K,1}(t).
\end{aligned}
\end{equation}
For $B_{414}$, we have
\begin{equation}\label{B414}
\begin{aligned}
|B_{414}|\lesssim&\frac{1}{\v}\|a^\v\|_{L^\infty_x}\|\p_x^\a\n_x\phi^\v\|_{L^2_x}\|\p_x^\a\n_x(\I-\P)g^\v\|_{L^2_{x,v}}+\frac{1}{\v}\|\p_x^\a a^\v\|_{L^2_x}\|\n_x\phi^\v\|_{L^\infty_x}\|\p_x^\a\n_x(\I-\P)g^\v\|_{L^2_{x,v}}\\
&+\frac{1}{\v}\sum_{1\leq|\a'|\leq|\a|-1}\|\p_x^{\a'} a^\v\|_{L^4_x}\|\p_x^{\a-\a'}\n_x\phi^\v\|_{L^4_x}\|\p_x^\a\n_x(\I-\P)g^\v\|_{L^2_{x,v}}\\
\lesssim&\frac{1}{\v}\|a^\v\|_{H^2_x}\|\n_x\phi^\v\|_{H^{k-1}_x}\|(\I-\P)g^\v\|_{\mathcal{H}^k_x\mathcal{L}^2_v}\\ 
\lesssim&\E^{\frac{1}{2}}_k(t)\D_k(t).
\end{aligned}
\end{equation}
Hence, for $B_{41}$, by collecting the estimates \eqref{B411}-\eqref{B414}, we have
\begin{equation}\label{B41}
\begin{aligned}
|B_{41}|-\frac{1}{8}\|(\p_x^\a\n_x a^\v,\p_x^\a\n_x\phi^\v)\|_{L^2_x}^2-2^{18}\D_{k,K,1}(t) \, \lesssim \, \E^{\frac{1}{2}}_k(t)\D_k(t).
\end{aligned}
\end{equation}

For $B_{42}$, we have, for $|\a'|=1$
\begin{equation}\label{B42}
\begin{aligned}
|B_{42}|\lesssim& \, \frac{1}{\v}\sum_{0\leq|\tilde{\a}|\leq|\a|}\|\p_x^{\tilde{\a}+\a'} b^\v\|_{L^2_x}\|\p_x^{\tilde{a}} b^\v\|_{L^4_x}\|\p_x^{\a-\tilde{\a}+\a'}\phi^\v\|_{L^4_x}\\
\lesssim& \, \frac{1}{\v}\|\n_x\phi^\v\|_{H^k_x}\|b^\v\|_{H^k_x}^2\\
\lesssim& \, \E^{\frac{1}{2}}_k(t)\D_k(t).
\end{aligned}
\end{equation}

For $B_{43}$, we have
\begin{equation}\label{B43}
\begin{aligned}
|B_{43}|\leq& \, \frac{2^8}{\v^2}\|b^\v\|_{H^k_x}\|(\I-\P)g^\v\|_{\mathcal{H}^k_x\mathcal{L}^2_v} \\ 
\leq& \, \frac{2^8}{\v^2}\big(\|b^\v\|_{H^k_x}^2+\|(\I-\P)g^\v\|_{\mathcal{H}^k_x\mathcal{L}^2_v}^2\big)\\
\leq& \, 2^8\D_{k,K,1}(t).
\end{aligned}
\end{equation}

Foe $B_{44}$, by noticing that
\begin{multline*}
\|\n_x\Gamma_{ij}[(\I-\P)g^\v]\|_{L^2_x}=\|\int_{\R^3}\n_x(\I-\P)g^\v(v_i v_j-\delta_{ij})\M \d v\|_{L^2_x} \\[4pt]
\leq \|\n_x(\I-\P)g^\v\|_{L^2_{x,v}}\|(v_iv
_j-\delta_{ij})\M\|_{L^2_v} \leq 2^4\|\n_x(\I-\P)g^\v\|_{L^2_{x,v}},
\end{multline*}
we find 
\begin{equation}\label{B44}
\begin{aligned}
|B_{44}|\leq& \, \frac{2^8}{\v^2}\|b^\v\|_{H^k_x}\|\n_x(\I-\P)g^\v\|_{\mathcal{H}^{k-1}_x\mathcal{L}^2_v} \\
\leq& \, \frac{2^8}{\v^2}\big(\|b^\v\|_{H^k_x}^2+\|(\I-\P)g^\v\|_{\mathcal{H}^k_x\mathcal{L}^2_v}^2\big)\\
\leq& 2^8\D_{k,K,1}(t).
\end{aligned}
\end{equation}

%%%%%%%%%%%%%%%%%%%%%%%%%%%%%%%%%%%%%%%%%%%%%%%%%%%%%%%%%%%%%%%%%%%

Note that, for any $0\leq|\a|\leq k-1$ and $|\a'|=1$,
\begin{equation}\label{partial-b-0}
\frac{1}{\v}\sum_{i,j=1}^3\|\p_x^\a(\p_{x_i} b_j^\v+\p_{x_j} b_i^\v)\|^2_{L^2_x} = \frac{2}{\v}(\|\p_x^{\a+\a'}b^\v\|^2_{L^2_x} + \|\p_x^\a\up{div}_x b^\v\|^2_{L^2_x} ),
\end{equation}
then substituting the estimates on $B_{41}$, $B_{42}$, $B_{43}$, $B_{44}$ into \eqref{partial-b-1}, it yields that, for $|\a'|=1$,
\begin{multline}\label{B5}
\frac{d}{d t} \sum_{i,j=1}^3 \int\p_x^\a(\p_{x_i} b_j^\v + \p_{x_j} b_i^\v)\c\p_x^\a\Gamma_{ij}[(\I-\P)g^{\v} ]\, dx + \frac{2}{\v} (\|\p_x^{\a+\a'}b^\v\|^2_{L^2_x} + \|\p_x^\a\up{div}_x b^\v\|^2_{L^2_x})\\
-\frac{1}{8}\|(\p_x^{\a}\n_x a^\v,\p_x^{\a+\a'}\n_x\phi^\v)\|^2_{L^2_x}
-2^{19}\D_{k,K,1}(t)\lesssim \E^{\frac{1}{2}}_k(t)\D_k(t).
\end{multline}

%%%%%%%%%%%%%%%%%%%%%%%%%%%%%%%%%%%%%%%%%%

On the other hand, considering the equation $\eqref{Pg-a-b}_2$, we have,
\begin{equation}\label{B6-0}
\begin{aligned}
\|\p_x^{\a}\n_x a^\v \|^2_{L^2_x}= &\sum_{i=1}^3\l\p_x^\a\p_{x_i} a,\p_x^\a\p_{x_i} a^\v\r_x\\
=&-\v\frac{\d}{\d t} \int \partial_x^{\a}\n_x a^\v\c\p_x^\a b^\v \d x + \underbrace{\v\sum_{i=1}^3\l\p_x^\a\p_{x_i} \p_t a^\v,\p_x^\a b_i^\v\r_x}_{B_{51}}\\
&\underbrace{-\sum_{i=1}^3\l\p_x^\a\p_{x_i}a^\v,\p_x^\a\p_{x_i}\phi^\v\r_x}_{B_{52}}
\underbrace{-\frac{1}{\v}\sum_{i=1}^3\l\p_x^\a\p_{x_i}a^\v,\p_x^\a b_i^\v\r_x}_{B_{53}} 
\underbrace{-\sum_{i=1}^3\l\p_x^\a\p_{x_i}a^\v,\p_x^\a(a^\v\p_{x_i}\phi^\v)\r_x}_{B_{54}}\\
&\underbrace{-\sum_{i,j=1}^3\l\p_x^\a\p_{x_i} a^\v,\p_x^\a\p_{x_i}(\I-\P)g^\v(v_iv_j-1)\M\r_{x,v}}_{B_{55}}
\end{aligned}
\end{equation}
for $0 \leq |\a| \leq 3$.\\
For $B_{51}$, by using $\eqref{Pg-a-b}_1$, we have
\begin{equation}\label{B51}
B_{51} = -\sum_{i=1}^3\l\p_x^\a \p_{x_i} \up{div}_x b^\v,\p_x^\a b_i^\v\r= \|\p_x^\a\up{div}_x b^\v\|^2_{L^2_x}.
\end{equation}
For $B_{52}$, by inserting $\eqref{The-VPFP-system}_2$, we have
\begin{equation}\label{B52}
B_{52} = -\sum_{i=1}^3\l\p_x^\a\p_{x_i}a^\v,\p_x^\a\p_{x_i}\phi^\v\r_x=\sum_{i=1}^3\l\p_x^\a\p_{x_i}\Delta_x\phi^\v,\p_x^\a\p_{x_i}\phi^\v\r_x=-\|\p_x^{\a+\a'}\n_x\phi^\v\|_{L^2_x}^2,
\end{equation}
for $|\a'|=1$.\\
For $B_{53}$ and $B_{55}$, we find
\begin{equation}\label{B53}
\begin{aligned}
|B_{53}|=& \ \Big| -\frac{1}{\v}\sum_{i=1}^3\l\p_x^\a\p_{x_i} a^\v, \p_x^\a b_i^\v\r_x \Big| \\
\leq& \ \frac{1}{8}\|\p_x^{\a}\n_x a^\v\|^2_{L^2_x} + \frac{2^4}{\v^2}\|\p_x^\a b^\v\|^2_{L^2_x} \\[5pt]
\leq& \ \frac{1}{8}\|\p_x^{\a}\n_x a^\v\|^2_{L^2_x} + 2^4\D_{k,K,1}(t),
\end{aligned}
\end{equation}
and
\begin{equation}\label{B55}
\begin{aligned}
|B_{55}| \leq&\ \frac{1}{8}\|\p_x^{\a+\a'} a^\v\|^2_{L^2_x} + \frac{2^8}{\v^2}\|(\I-\P)g^\v\|^2_{\mathcal{H}^k_x\mathcal{L}^2_v} \\[5pt]
\leq&\ \frac{1}{8}\|\p_x^{\a+\a'} a^\v\|^2_{L^2_x} + 2^8\D_{k,K,1}(t).
\end{aligned}
\end{equation}
For $B_{54}$, we have
\begin{equation}\label{B54}
\begin{aligned}
|B_{54}|\lesssim &\sum_{0\leq|\tilde{\a}|\leq|\a|}\|\p_x^{\tilde{\a}}\n_x a^\v\|_{L^2_x}\|\p_x^{\tilde{\a}} a^\v\|_{L^4_x}\|\p_x^{\a-\tilde{\a}}\n_x\phi^\v\|_{L^4_x}\\[5pt]
\lesssim& \|g^\v\|_{H^k_x L^2_v}\|\n_x\phi^\v\|_{H^k_x}\|\n_x a^\v\|_{H^{k-1}_x}\\[5pt]
\lesssim & \E^{\frac{1}{2}}_k(t)\D_k(t).
\end{aligned}
\end{equation}
Then, inserting the estimates of $B_{51}$-$B_{54}$, i.e., \eqref{B51}-\eqref{B54} into \eqref{B6-0}, we obtain, for $|\a'|=1$,
\begin{equation}\label{B6}
\frac{3}{4}\|(\p_x^{\a}\n_x a^\v,\p_x^{\a+\a'}\n_x\phi^\v)\|^2_{L^2_x} + \v\frac{\d}{\d t}\int\p_x^{\a}\n_x a\c\p_x^\a b^\v \, \d x-\|\p_x^\a\up{div}_x b^\v\|_{L^2_x}^2-2^9\D_{k,K,1}(t) \lesssim\E^{\frac{1}{2}}_k(t)\D_k(t).
\end{equation}

Finally, combining \eqref{Energy-of-a-and-b}, \eqref{B5} and \eqref{B6}, we have,
\begin{multline}\label{B5-B6}
\frac{\d}{\d t}\sum_{i,j=1}^3\int\p_x^\a(\p_{x_i} b_j^\v+\p_{x_j} b_i^\v)\c\p_x^\a\Gamma_{ij}[(\I-\P)g^\v] \,\d x + \v\frac{\d}{\d t}\int\p_x^{\a}\n_x a^\v\c\p_x^\a b^\v\,\d x\\
+\frac{1}{\v} \big(\|\p_x^{\a}\n_x b^\v\|^2_{L^2_x} + \|\p_x^\a\up{div}_x b^\v\|^2_{L^2_x} \big)
+ \frac{1}{2}\|(\p_x^{\a}\n_x a^\v,\p_x^{\a+\a'}\n_x\phi^\v)\|^2_{L^2_x}-2^{20}\D_{k,K,1}(t)
\lesssim\E^{\frac{1}{2}}_k(t)\D_k(t),
\end{multline}
for any $0\leq |\a|\leq k-1$ and $|\a'|=1$.
The proof of Lemma \ref{Estiamte-Mi-fluid-part} can be completed by summing up $0 \leq |\a| \leq k-1$ in \eqref{B5-B6}.

\end{proof}

\subsection{Proof of Proposition \ref{total_energy_dissipation}}
\label{subsec:proof_total_energy}

In this subsection, we present how to combine the Lemmas \ref{Estimate-kinetic-Mi-1}, \ref{Estiamte-Mi-kinetic-part-mixed-partial} and \ref{Estiamte-Mi-fluid-part} together to obtain the total energy estimate (Proposition \ref{total_energy_dissipation}).

\begin{proof}
We prove the total energy estimate \eqref{total_energy_estimate} in the following three steps:

\textbf{Step 1:} Choosing a constant $\tilde{\lambda}_1>0$ large enough such that
\begin{equation}\label{tilde-lambda-1}
C_1 \tilde{\lambda}_1 \geq \frac{\tilde{C}_3}{2},
\end{equation}
where $C_1,\,\tilde{C}_3$ are constants in \eqref{Step-one} of Lemma \ref{Estimate-kinetic-Mi-1} and \eqref{last-step} of Lemma \ref{Estiamte-Mi-fluid-part},  respectively.
Then, by applying \eqref{last-step} $ + \tilde{\lambda}_1\times$ \eqref{Step-one}, we can find that there exists a constant $C_4>0$ such that
\begin{equation}\label{First-time}
    \frac{1}{2}\frac{\d}{\d t}\big(\E_F(t)+\tilde{\lambda}_1\E_{k,K,1}(t)\big)+C_4\big(\D_{k,K,2}(t)+\D_{k,F}(t)\big) 
    \lesssim\E^{\frac{1}{2}}_k(t)\D_k(t).
\end{equation}

\textbf{Step 2:} Choosing a constant $\tilde{\lambda}_2>0$ large enough such that
\begin{equation}\label{tilde-lambda-2}
C_4 \tilde{\lambda}_2 \geq \frac{\tilde{C}_2}{2},
\end{equation}
where $\tilde{C}_2$ is the constant in \eqref{Step-two} of Lemma \ref{Estiamte-Mi-kinetic-part-mixed-partial} and $C_4$ is the constant in \eqref{First-time}.
Then, by applying \eqref{Step-one}$+\tilde{\lambda}_2\times$\eqref{First-time}, there exists a constant $ C_5 >0$ such that
\begin{equation}
\frac{1}{2}\frac{\d}{\d t}\big(\tilde{\lambda}_1\tilde{\lambda}_2\E_{k,K,1}(t)+\E_{k,K,2}(t)+\tilde{\lambda}_1\tilde{\lambda}_2\E_{k,F}(t)\big) + C_5\big(\sum_{i=1}^2\D_{k,K,i}(t)+\D_{k,F}(t)\big)
\lesssim \E^{\frac{1}{2}}(t)\D(t).
\end{equation}

\textbf{Step 3:} Denoting
\begin{equation}
\begin{aligned}
\E_k(t):=&\tilde{\lambda}_1\tilde{\lambda}_2\E_{k,K,1}(t)+\E_{k,K,2}(t)+\tilde{\lambda}_1\tilde{\lambda}_2\E_{k,F}(t),\\
\D_k(t):=&\sum_{i=1}^2 \D_{k,K,i}(t)+\D_F(t).
\end{aligned}
\end{equation}
and re-naming $\tilde{C}=C_5$, we finally obtain the total energy estimate \eqref{total_energy_estimate}:
\begin{equation}
\frac{1}{2}\frac{\d}{\d t}\E_k(t)+\tilde{C}\D_k(t) \,\lesssim\, \E^{\frac{1}{2}}_k(t)\D_k(t),
\end{equation}
where $\lambda_i, i =1,2,3$, are defined as follows:
\begin{equation}\label{lambda-constants}
\quad \lambda_1=\tilde{\lambda}_1\tilde{\lambda}_2,
\quad
\lambda_2=1,
\quad
\lambda_3=\tilde{\lambda}_1\tilde{\lambda}_2.
\end{equation}
\end{proof}

%of $\p_t u$, $\n_x\rho$, and fluid part $(a,b)$
%Recalling all \textit{priori} estimatess as follow
%\begin{equation}\label{Macro-part}
%\frac{1}{2}\frac{d}{dt}\E_{ma}(t)+C_{3}\D_{ma}(t)\leq 0,
%\end{equation}

%\begin{equation}\label{Step-one}
%\begin{aligned}
%\frac{1}{2}\frac{d}{dt}\E_{mi,K,1}(t)+C_{1,1}\D_{mi,K,1}(t)\leq C\big(\D_{ma}(t)+\sum_{i=1}^5(\v+\sqrt{\delta})^i\E^{\frac{i}{2}}(t)\D(t)\big),
%\end{aligned}
%\end{equation}

%\begin{equation}\label{Step-two}
%\begin{aligned}
%\frac{1}{2}\frac{d}{dt}\E_{mi,K,2}(t)+C_{1,2}\D_{mi,K,2}(t)\leq C\Big(\D_{mi,K,1}(t)+\D_{mi,K,3}(t)+\sum_{i=1}^4(\v+\sqrt{\delta})^i\E^{\frac{i}{2}}(t)\D(t)\Big),
%\end{aligned}
%\end{equation}

%\begin{equation}\label{Step three}
%\begin{aligned}
%\frac{1}{2}\frac{d}{dt}\E_{mi,K,3}(t)+C_{1,3}\D_{mi,K,3}(t)\leq C\big(\D_{mi,K,1}(t)+\sum_{i=1}^6(\v+\sqrt{\delta})^i\E^{\frac{i}{2}}(t)\D(t)\big),
%\end{aligned}
%\end{equation}

%\begin{equation}\label{Step four}
%\begin{aligned}
%\frac{1}{2}\frac{d}{dt}\E_{mi,K,4}(t)+C_{1,4}\D_{mi,K,4}(t)\leq C\big(\D_{mi,K,1}(t)+\D_{mi,F}(t)+\D_{ma}(t)+(\v+\sqrt{\delta})\E^{\frac{1}{2}}(t)\D(t)\big),
%\end{aligned}
%\end{equation}
%and
%\begin{equation}\label{last-step}
%\begin{aligned}
%\frac{1}{2}\frac{d}{dt}\E_{mi,F}(t)+C_{1,5}\D_{mi,F}(t)\leq C\Big(\D_{mi,K,1}(t)+\D_{ma}(t)+\sum_{i=1}^2(\v+\sqrt{\delta})^i\E^{\frac{i}{2}}(t)\D(t)\big).
%\end{aligned}
%\end{equation}
%%%%%%%%%%%%%%%%%%%%%%%%%%%%%%%%%%%%%%%%%%%%%%%%%%%%%%%%%%%%%%%%%%%%%%%%%

%--------------------------------------------------------------------------------

\subsection{Proof of Theorem \ref{Global-in-time-solution-of-VPFP}}
\label{subsec:proof-global}
In this subsection, we present the key part of the proof of Theorem \ref{Global-in-time-solution-of-VPFP}. In fact, The global well-posedness of $(g^\v, \nabla_x \phi^\v)$ for the VPFP system \eqref{The-VPFP-system}, as stated in Theorem \ref{Global-in-time-solution-of-VPFP}, directly follows from the local well-posedness result (Proposition \ref{Local-in-time}) combined with a standard continuity argument. The crucial ingredient for extending the local solution globally is the uniform energy estimate established in Proposition \ref{total_energy_dissipation}. For completeness, we refer the reader to \cite{JL22} for further details on this methodology.

Therefore, we only illustrate that the energy functional $\mathcal{E}_k(t)$ is continuous in $[0, T^*]$, where $T^*$ is given in Proposition \ref{Local-in-time}.
First, for any $0<\v\leq 1$, we have
\begin{equation*}
\frac{1}{C_4}\mathbb{E}_k(t) \leq \mathcal{E}_k(t) \leq C_4\mathbb{E}_k(t), \quad \frac{1}{C_4}\mathbb{D}_k(t) \leq \mathcal{D}_k(t) \leq C_4\mathbb{D}_k(t),\,
\end{equation*}
holds for any $t\in[0,T^*]$, where the constant $C_4 > 0$ is independent of $\v$ and $T^*$.

Furthermore, by considering the energy estimates \eqref{The-total-energy-of-local-in-time-solution}, \eqref{total_energy_estimate} and the assumption $\mathbb{E}(0)\leq\delta_0$ in Theorem \ref{Global-in-time-solution-of-VPFP} , we find, for any $[t_1,t_2]\subset [0,T^*]$ and $0<\v\leq 1$,
\begin{multline*}
\Big|\mathcal{E}_k(t_2)-\mathcal{E}_k(t_1)\Big| \lesssim  \int_{t_1}^{t_2}\E_k^{\frac{1}{2}}(t)\D_k(t) \,\d t 
\lesssim  \sup_{0\leq t\leq T^*}\E_k^{\frac{1}{2}}(t)\int_{t_1}^{t_2}\D_k(t) \,\d t\\[5pt]
%\lesssim  \E_k^{\frac{1}{2}}(0)\int_{t_1}^{t_2}\D_k(t) \,\d t\\[5pt]
\lesssim  \mathbb{E}^{\frac{1}{2}}(0)\int_{t_1}^{t_2}\D_k(t) \,\d t
\lesssim \sqrt{\delta_0}\int_{t_1}^{t_2}\D_k(t) \,\d t\to\,0,\quad \up{as}\quad t_1\to t_2,
\end{multline*}
which implies the continuity of $\E_k(t)$ in $t\in[0,T^*]$.

%%%%%%%%%%%%%%%%%%%%%%%%%%%%%%%%%%%%%%%%%%%%%%%%%%%%%%%%%%%%%%%%%%%%%%%%%
\section{Rigorous justification of the hydrodynamic limit (Theorem \ref{Limit-Fluid-equations})}
\label{sec:limit}

In this section, by following \cite{JNXCJZHJ18}, we will provide a rigorous justification of the limiting process from the scaled VPFP system \eqref{The-VPFP-system} to the DDP system \eqref{The-Drift-Diffusion-Possion-system} as $\v \to 0$, i.e., the proof of Theorem \ref{Limit-Fluid-equations}.

%------------------------------------------------------------------
\subsection{Compactness from the uniform energy estimates}
\label{subsec:compactness}

By the uniform energy estimate \eqref{Uniform energy estimate} in Theorem \ref{Global-in-time-solution-of-VPFP}, there exists a constant $C > 0$, independent of $\v$, such that for any $0<\v\leq 1$ and $k\geq 3$,
\begin{equation}\label{Bound of g}
\begin{aligned}
\sup_{t\geq 0}\big(\|g^\v \|_{H^k_x L^2_v}^2+\|\n_x\phi^\v\|_{H^k_x}^2\big)\leq C,
\end{aligned}
\end{equation}
and
\begin{equation}\label{The-energy-dissipation-bound-of-g}
\begin{aligned}
\int_0^T\|(\I-\P_0)g^\v\|_{\mathcal{H}^k_{x,v}}^2 \,\d t \leq C\v^2,
\end{aligned}
\end{equation}
for any given $T>0$.

From \eqref{Bound of g}, we can find that there exist $g_0 \in L^\infty\big(0,+\infty;H^k_{x,v}\big)$ and $\n_x\phi_0\in L^\infty\big(0,+\infty;H^k_x\big)$ such that
\begin{equation}\label{Convergence-of-g-and-phi}
\begin{aligned}
g^\v(t,x,v) & \to g_0(t,x,v),\quad \up{weakly-$\star$ for $t\in[0,T]$, strongly in $H^{k-1}_x$, weakly in $H^k_v$},\\[5pt]
\n_x\phi^\v(t,x) & \to \n_x\phi_0(t,x) ,\quad \up{weakly-$\star$ for $t\in[0,T]$, strongly in $H^{k-1}_x$},\\
\end{aligned}
\end{equation}
as $\v \to 0$ for any given $T>0$. 

From \eqref{The-energy-dissipation-bound-of-g}, we have
\begin{equation}\label{Convergence of of I-P g}
(\I-\P_0)g^\v (t,x,v) \to 0, \quad \up{in}\quad L^2 \big(0,T;\mathcal{H}^k_{x,v}\big),
\end{equation}
as $ \v \to 0$ for any given $T>0$. 

Combining the convergence of \eqref{Convergence-of-g-and-phi} and \eqref{Convergence of of I-P g}, it yields that
\begin{equation}\label{I-P of g0}
(\I-\P_0)g_0(t,x,v) = 0,
\end{equation}
which implies the existence of $\rho_0\in L^\infty \big(0,T;H^k_x\big)$ such that
\begin{equation}\label{The form of g0}
g_0(t,x,v) = \rho_0(t,x)\M,
\end{equation}
for any given $T>0$. 
%----------------------------------------------------------------
\subsection{Justification of the limiting process}
\label{subsec:justification}

Applying the convergence of $g^\v \to g_{0}$ in \eqref{Convergence-of-g-and-phi} and recalling \eqref{The-def-a}, we have
\begin{equation}\label{The convergence of rho-u-theta}
a^\v\to \rho_0, \quad \up{weakly-$\star$ for $t\in[0,T]$, strongly in $H^{k-1}_x$},
\end{equation}
as $\v \to 0$.

Next, multiplying \eqref{The-VPFP-system} by $\M$ and integrating over $v$, it leads to
\begin{equation}\label{The-local-conservation-laws}
\left\{
\begin{aligned}
&\p_t a^\v+\frac{1}{\v}\up{div}_x\l g^\v,v\M\r_v=0,\\[5pt]
&-\Delta_x\phi^\v=a^\v.
\end{aligned}
\right.
\end{equation}

Using $\eqref{The-local-conservation-laws}_1$, \eqref{The-energy-dissipation-bound-of-g} and \eqref{Bound of g}, for any given $T>0$, we have
\begin{equation}\label{Bounded-with-pa}
    \begin{aligned}
        \int_0^T\|\p_t a^\v\|_{H^{k-1}_x}^2\d t=&\frac{1}{\v^2}\int_0^T\|\up{div}\l (\I-\P_0)g^\v,v\M\r_v\|_{H^{k-1}_x}^2 \,\d t\\[5pt]
        \lesssim&\frac{1}{\v^2}\int_0^T\|(\I-\P_0)g^\v\|_{H^{k}_x L^2_v}^2 \,\d t\\[5pt]
        \lesssim&\,1,
    \end{aligned}
\end{equation}
and 
\begin{equation}\label{Bounded-with-a}
    \begin{aligned}
     \sup\limits_{t\in[0,T]}\|a^\v\|_{H^k_x}=&\sup\limits_{t\in[0,T]}\|\l g^\v,\M\r_v\|_{H^k_x}\leq\sup\limits_{t\in[0,T]}\| g^\v\|_{H^k_xL^2_v}\lesssim\,1.
    \end{aligned}
\end{equation}
Then, by noting the convergence \eqref{The convergence of rho-u-theta}, and the estimates \eqref{Bounded-with-pa}, \eqref{Bounded-with-a}, we can obtain by the Aubin-Lions Lemma that 
\begin{equation*}
    \rho_0\in L^\infty(0,T;H^k_x)\cap C([0,T];H^{k-1}_x),
\end{equation*}
such that
\begin{equation}\label{The-convergence-of-a}
a^\v  \to \rho_0 \quad \up{strongly in $C([0,T];H^{k-1}_x)$},
\end{equation}
as $\v\to 0$.

The similar argument can also be extended to $\Delta_x\phi^\v$. Using \eqref{The-local-conservation-laws} and \eqref{Bounded-with-a},  we have, for any given $T>0$,
\begin{equation}\label{Bounded-with-p-phi}
    \begin{aligned}
        \int_0^T\|\p_t\n_x^2\phi^\v\|_{H^{k-1}_x}^2\d t=\int_0^T\|\p_t\Delta_x\phi^\v\|_{H^{k-1}_x}^2\d t=&\int_0^T\|\p_t a^\v\|_{H^{k-1}_x}^2\d t
        \lesssim\,1,
    \end{aligned}
\end{equation}
and the Poincar$\acute{\up{e}}$ inequality for $\p_t\n_x\phi^\v$ that
\begin{equation}\label{Poincare-for-pt-phi}
\int_0^T\int_{\T^3}|\p_t\n_x\phi^\v|^2 \,\d x \,\d t \,\lesssim\,  \int_0^T\int_{\T^3}|\p_t\n_x^2\phi^\v|^2 \,\d x \,\d t.  
\end{equation}
Then, considering the convergence $\eqref{Convergence-of-g-and-phi}_2$ and estimates \eqref{Bound of g}, \eqref{Bounded-with-p-phi}, \eqref{Poincare-for-pt-phi}, we can obtain, from the Aubin-Lions Lemma, that
\begin{equation*}
    \n_x\phi_0\in L^\infty\big(0,T;H^k_x\big) \cap  C\big([0,T];H^k_x\big),
\end{equation*}
such that
\begin{equation}\label{The-convergence-of-nx-phi}
\n_x\phi^\v\to\n_x\phi_0, \quad \up{strongly in $C\big([0,T];H^{k-1}_x\big)$},
\end{equation}
as $\v\to 0$.

Furthermore, according to \eqref{The-formally-limits-2}, we have
\begin{equation}\label{Def-R}
    \begin{aligned}
      \frac{1}{\v} &\up{div}_x\l g^\v,v\M\r_v=-\Delta_x\rho_0+\up{div}\big[(\rho_0+1)\n_x\phi_0\big]
      \underbrace{-\up{div}_x \l\v\p_t g^\v,v\M\r_v}_{R_1}\underbrace{-\up{div}_x\l v\c\n_x(g^\v-g_0),v\M\r_v}_{R_2}\\[5pt]
      &\underbrace{-\up{div}_x\l v\c\n_x(\phi^\v-\phi_0),v\M\r_v}_{R_3}\underbrace{-\up{div}_x\l\n_x\phi^\v\c\n_v(g^\v\M)-\n_x\phi_0\c\n_v(g_0\M),v\r_v}_{R_4}.
    \end{aligned}
\end{equation}

From the energy estimate \eqref{Bound of g}, we find, for any test functions $\varphi(t,x)\in C^\infty_0([0,+\infty)\times\T^3)$,
\begin{equation}
    \begin{aligned}
        \Big|\int_0^{+\infty}\int_{\T^3}R_1(t,x)\varphi(t,x) \,\d x \,\d t \Big| = &\, \v\Big|\int_0^{+\infty}\int_{\T^3}\up{div}_x\l g^\v,v\M\r_v\p_t\varphi \,\d x \,\d t\Big|\\[5pt]
        \lesssim&\, \v\|\n_x g^\v\|_{L^\infty_t L^2_{x,v}}\|\p_t\varphi\|_{L^1_t L^2_x}\to\,0,
    \end{aligned}
\end{equation}
which implies that
\begin{equation}\label{R1}
 R_1= -\v\p_t\up{div}_x \l g^\v,v\M\r_v \rightharpoonup \,0, 
\end{equation}
as $\v\to 0$.

For $R_2$ and $R_3$, by using \eqref{Convergence-of-g-and-phi}, we have, for $k\geq 3$,
\begin{equation}\label{R2-R3}
\begin{aligned}
 \|R_2\|_{H^{k-3}_x}\lesssim \|g^\v-g_0\|_{H^{k-1}_{x}L^2_v}\to 0, \quad &\up{weakly-$\star$ for $t\geq 0$}\\[5pt]
 \|R_3\|_{H^{k-2}_x}\lesssim\|\n_x\phi^\v-\n_x\phi_0\|_{H^{k-1}_{x}L^2_v}\to 0, \quad \quad &\up{weakly-$\star$ for $t\geq 0$}
\end{aligned}
\end{equation}
as $\v \to 0$.

For $R_4$, we have
\begin{equation}\label{R4}
    \begin{aligned}
     \|R_4\|_{H^{k-3}_x}=&\|\up{div}_x \big\l\n_x\phi^\v g^\v-\n_x\phi_0 g_0,\M \big\r_v \|_{H^{k-1}_x}\\[5pt]
     \leq&\|\up{div}_x\l\n_x(\phi^\v-\phi_0)g^\v,\M\r_v\|_{H^{k-3}_x}+\|\up{div}_x\l\n_x\phi_0(g^\v-g_0),\M\r_v\|_{H^{k-3}_x}\\[5pt]
     \lesssim&\|\n_x(\phi^\v-\phi_0)\|_{H^{k-1}_x}\|g^\v\|_{H^{k-1}_x L^2_v}+\|\n_x\phi_0\|_{H^{k-1}_x}\|g^\v-g_0\|_{H^{k-1}_x L^2_v}\\[5pt]
     \to&\,0 \quad \up{weakly-$\star$ for $ 0 \leq t \leq T$,}
    \end{aligned}
\end{equation}
as $\v\to 0$.

For $\p_t a^\v$, by noting \eqref{The convergence of rho-u-theta} and \eqref{The-convergence-of-a}, we can follow the similar argument as in\eqref{R1} to obtain
\begin{equation}\label{pa}
    \begin{aligned}
     \p_t a^\v \rightharpoonup \p_t \rho_0,
    \end{aligned}
\end{equation}
as $\v\to 0$.

For $\Delta_x\phi^\v$ in $\eqref{The-local-conservation-laws}_2$, using $\eqref{Convergence-of-g-and-phi}_2$, we have
\begin{equation}\label{Delta-phi}
 -\Delta_x\phi^\v\to-\Delta_x\phi_0, \quad \up{weakly-$\star$ for $ 0 \leq t \leq T$, strongly in $H^{k-2}_x$,}
\end{equation}
as $\v\to 0$.

Therefore, we obtain, for any $k\geq 3$,
\begin{equation*}
\rho_0\in L^\infty\big(0,T;H^k_x\big) \cap C\big([0,T];H^{k-1}_x\big),
\quad \n_x\phi_0\in L^\infty\big(0,T;H^k_x\big)\cap C\big([0,T];H^k_x\big),    
\end{equation*}
for any given $T>0$ satisfying the DDP system 
\begin{equation*}
\left\{
\begin{aligned}
&\p_t\rho_0=\Delta_x\rho_0+\up{div}_x\big[(\rho_0+1)\n_x\phi_0\big],\\[4pt]
&-\Delta_x\phi_0=\rho_0,
\end{aligned}
\right.
\end{equation*}
with the initial conditions
\begin{equation*}
  \rho_0(0,x)=\rho_0^{in}(x), \quad \n_x\phi_0(0,x)=\n_x\phi^{in}_0(x),  
\end{equation*}
where the uniqueness can be further derived by the stability energy estimate in the higher-regularity spaces.
%%%%%%%%%%%%%%%%%%%%%%%%%%%%%%%%%%%%%%%%%%%%%%%%%%%%%%%%%

\subsection{Proof of Corollary \ref{Conergence}}
\label{subsec:proof_main_corollary}

In this subsection, we finally complete the proof of the convergence in Corollary \ref{Conergence} by using the embedding theorem.
For any given $T>0$ and $k\geq 4$, we have
\begin{equation}
\begin{aligned}
\int_0^T|f^\v(t,x,v)-(1+\rho_0(t,x))M|^2\d t\leq&\int_0^T\|(g^\v-\rho_0\M)\M\|_{L^\infty_xL^\infty_v}^2\d t\\[5pt]
\lesssim &\int_0^T\|(a^\v-\rho_0)\M+(\I-\P_0)g^\v\|_{L^\infty_xL^\infty_v}^2 \,\d t\\[5pt]
\lesssim &\int_0^T\|a^\v-\rho_0\|_{H^2_x}^2\d t+\int_0^T\|(\I-\P_0)g^\v\|_{H^2_x H^2_v}^2\d t\\[5pt]
\lesssim &\, T\|a^\v-\rho_0\|_{C([0,T];H^2_x)}^2+\v^2,\\[5pt]
\to&\,0, \quad \up{as $\v \to 0$},
\end{aligned}
\end{equation}
where the expansion \eqref{def:perturbation-form} is used in the first inequality, the decomposition \eqref{MM} is used in the second inequality, the Sobolev embedding $H^2\hookrightarrow L^\infty$ is applied in the third inequality.
\begin{equation}
\begin{aligned}
|\n_x\phi^\v(t,x)-\n_x\phi_0(t,x)|\leq&\sup\limits_{t\in[0,T]}\|\n_x\phi^\v-\n_x\phi_0\|_{L^\infty_x}\\[5pt]
\lesssim&\sup\limits_{t\in[0,T]}\|\n_x\phi^\v-\n_x\phi_0\|_{H^2_x}\\[5pt]
\to&\,0, \quad \up{as $\v \to 0$},
\end{aligned}    
\end{equation}
where the Sobolev embedding $H^2\hookrightarrow L^\infty$ is applied in the second inequality.
The proof of convergence is finally completed by the uniform boundedness as in the energy estimate \eqref{The-energy-dissipation-bound-of-g} as well as \eqref{The-convergence-of-a}, \eqref{The-convergence-of-nx-phi}.

\section*{Acknowledgment}
ZF was partially supported by the NSFC grant (No.12201140), and Guangzhou Basic and Applied Basic Research Foundation (No.2025A04J0029).
KQ acknowledges support from AMS-Simons Travel Award grant, and part of this work is completed and based upon work supported by the National Science Foundation under Grant No.~DMS-2424139, while KQ was in residence at the Simons Laufer Mathematical Sciences Institute in Berkeley, California, during the Fall 2025 semester.

\bibliographystyle{siam}
\bibliography{Kinetic-Fluid_Project}

@article {ACGS2001,
    AUTHOR = {Arnold, A. and Carrillo, J. A. and Gamba, I. and Shu, C.-W.},
     TITLE = {Low and high field scaling limits for the {V}lasov- and
              {W}igner-{P}oisson-{F}okker-{P}lanck systems},
      NOTE = {The Sixteenth International Conference on Transport Theory,
              Part I (Atlanta, GA, 1999)},
   JOURNAL = {Transport Theory Statist. Phys.},
  FJOURNAL = {Transport Theory and Statistical Physics},
    VOLUME = {30},
      YEAR = {2001},
    NUMBER = {2-3},
     PAGES = {121--153},
      ISSN = {0041-1450,1532-2424},
   MRCLASS = {82D37 (82C40)},
  MRNUMBER = {1848592},
MRREVIEWER = {Vittorio\ Romano},
       DOI = {10.1081/TT-100105365},
       URL = {https://doi.org/10.1081/TT-100105365},
}

@article {NPS2001,
    AUTHOR = {Nieto, Juan and Poupaud, Fr\'ed\'eric and Soler, Juan},
     TITLE = {High-field limit for the {V}lasov-{P}oisson-{F}okker-{P}lanck
              system},
   JOURNAL = {Arch. Ration. Mech. Anal.},
  FJOURNAL = {Archive for Rational Mechanics and Analysis},
    VOLUME = {158},
      YEAR = {2001},
    NUMBER = {1},
     PAGES = {29--59},
      ISSN = {0003-9527,1432-0673},
   MRCLASS = {82D10 (35B40 82C31)},
  MRNUMBER = {1834113},
MRREVIEWER = {Yan\ Guo},
       DOI = {10.1007/s002050100139},
       URL = {https://doi.org/10.1007/s002050100139},
}

@article {Degond1986,
    AUTHOR = {Degond, Pierre},
     TITLE = {Global existence of smooth solutions for the
              {V}lasov-{F}okker-{P}lanck equation in {$1$} and {$2$} space
              dimensions},
   JOURNAL = {Ann. Sci. \'Ecole Norm. Sup. (4)},
  FJOURNAL = {Annales Scientifiques de l'\'Ecole Normale Sup\'erieure.
              Quatri\`eme S\'erie},
    VOLUME = {19},
      YEAR = {1986},
    NUMBER = {4},
     PAGES = {519--542},
      ISSN = {0012-9593},
   MRCLASS = {35Q20 (76W05)},
  MRNUMBER = {875086},
MRREVIEWER = {P.\ L.\ Sulem},
}

@article {BostanGoudon2008,
    Author= {Bostan, Mihai and Goudon, Thierry},
     Title= {Low field regime for the relativistic
              {V}lasov-{M}axwell-{F}okker-{P}lanck system; the one and one
              half dimensional case},
   Journal = {Kinet. Relat. Models},
  Fjournal = {Kinetic and Related Models},
    Volume = {1},
      Year = {2008},
    Number = {1},
     Pages = {139--170},
      Issn = {1937-5093,1937-5077},
   Mrclass = {35Q75 (35B40 35K55 82D10)},
  Mrnumber = {2383720},
       Doi = {10.3934/krm.2008.1.139},
       Url = {https://doi.org/10.3934/krm.2008.1.139},
}

@article {FangQiWen2024,
    Author = {Fang, Zhendong and Qi, Kunlun and Wen, Huanyao},
     Title = {The small {D}eborah number limit for the fluid-particle flows:
              incompressible case},
   Journal = {Math. Models Methods Appl. Sci.},
  FJournal = {Mathematical Models and Methods in Applied Sciences},
    Volume = {34},
      Year= {2024},
    Number = {12},
     Pages = {2265--2304},
      Issn= {0218-2025,1793-6314},
   Mrclass = {35Q35 (35B25 35B40 35Q30 35Q83 35Q84 82C40)},
  Mrnumber = {4818560},
       Doi= {10.1142/S0218202524500489},
       Url = {https://doi.org/10.1142/S0218202524500489},
Mrreviewer ={Lei, yuanjie},
}

@article {GoudonNietoPoupaudSoler2005,
    AUTHOR = {Goudon, T. and Nieto, J. and Poupaud, F. and Soler, J.},
     TITLE = {Multidimensional high-field limit of the electrostatic
              {V}lasov-{P}oisson-{F}okker-{P}lanck system},
   JOURNAL = {J. Differential Equations},
  FJOURNAL = {Journal of Differential Equations},
    VOLUME = {213},
      YEAR = {2005},
    NUMBER = {2},
     PAGES = {418--442},
      ISSN = {0022-0396,1090-2732},
   MRCLASS = {35Q99 (35B25 35F25 82D10)},
  MRNUMBER = {2142374},
MRREVIEWER = {Hidetoshi\ Tahara},
       DOI = {10.1016/j.jde.2004.09.008},
       URL = {https://doi.org/10.1016/j.jde.2004.09.008},
}

@article {CJADRJMA11,
    AUTHOR = {Carrillo, Jos\'{e} A. and Duan, Renjun and Moussa, Ayman},
     TITLE = {Global classical solutions close to equilibrium to the {V}lasov-{F}okker-{P}lanck-{E}uler system},
   JOURNAL = {Kinet. Relat. Models},
  FJOURNAL = {Kinetic and Related Models},
    VOLUME = {4},
      YEAR = {2011},
    NUMBER = {1},
     PAGES = {227--258},
      ISSN = {1937-5093,1937-5077},
   MRCLASS = {35Q83 (35A01 35A09 35B10 35B30 35Q84 76T20 82C40)},
  MRNUMBER = {2765745},
MRREVIEWER = {Xiaoming\ Wang},
       DOI = {10.3934/krm.2011.4.227},
       URL = {https://doi.org/10.3934/krm.2011.4.227},
}

@article {CMKKLJ11,
    AUTHOR = {Chae, Myeongju and Kang, Kyungkeun and Lee, Jihoon},
     TITLE = {Global existence of weak and classical solutions for the {N}avier-{S}tokes-{V}lasov-{F}okker-{P}lanck equations},
   JOURNAL = {J. Differential Equations},
  FJOURNAL = {Journal of Differential Equations},
    VOLUME = {251},
      YEAR = {2011},
    NUMBER = {9},
     PAGES = {2431--2465},
      ISSN = {0022-0396,1090-2732},
   MRCLASS = {35Q35 (35A01 35A09 35D30 76B03 76D03)},
  MRNUMBER = {2825335},
       DOI = {10.1016/j.jde.2011.07.016},
       URL = {https://doi.org/10.1016/j.jde.2011.07.016},
}

@article {LHLLSQYT22,
    AUTHOR = {Li, Hai-Liang and Liu, Shuangqian and Yang, Tong},
     TITLE = {The {N}avier-{S}tokes-{V}lasov-{F}okker-{P}lanck system in bounded domains},
   JOURNAL = {J. Stat. Phys.},
  FJOURNAL = {Journal of Statistical Physics},
    VOLUME = {186},
      YEAR = {2022},
    NUMBER = {3},
     PAGES = {Paper No. 42, 32},
      ISSN = {0022-4715,1572-9613},
   MRCLASS = {35Q84 (35Q35 82C31)},
  MRNUMBER = {4377286},
       DOI = {10.1007/s10955-022-02886-7},
       URL = {https://doi.org/10.1007/s10955-022-02886-7},
}

@article {LFCMYMWDH17,
    AUTHOR = {Li, Fucai and Mu, Yanmin and Wang, Dehua},
     TITLE = {Strong solutions to the compressible {N}avier-{S}tokes-{V}lasov-{F}okker-{P}lanck equations: global
              existence near the equilibrium and large time behavior},
   JOURNAL = {SIAM J. Math. Anal.},
  FJOURNAL = {SIAM Journal on Mathematical Analysis},
    VOLUME = {49},
      YEAR = {2017},
    NUMBER = {2},
     PAGES = {984--1026},
      ISSN = {0036-1410,1095-7154},
   MRCLASS = {35Q35 (35B40 35D35 76D03 76D05)},
  MRNUMBER = {3628310},
MRREVIEWER = {Alberto\ Valli},
       DOI = {10.1137/15M1053049},
       URL = {https://doi.org/10.1137/15M1053049},
}

@article {MYMWDH20,
    AUTHOR = {Mu, Yanmin and Wang, Dehua},
     TITLE = {Global well-posedness and optimal large-time behavior of strong solutions to the non-isentropic particle-fluid flows},
   JOURNAL = {Calc. Var. Partial Differential Equations},
  FJOURNAL = {Calculus of Variations and Partial Differential Equations},
    VOLUME = {59},
      YEAR = {2020},
    NUMBER = {4},
     PAGES = {Paper No. 110, 42},
      ISSN = {0944-2669,1432-0835},
   MRCLASS = {76N06 (35B30 35B40 35D35 35Q35)},
  MRNUMBER = {4111819},
MRREVIEWER = {Xin\ Liu},
       DOI = {10.1007/s00526-020-01776-8},
       URL = {https://doi.org/10.1007/s00526-020-01776-8},
}

@article {MAVA07,
    AUTHOR = {Mellet, Antoine and Vasseur, Alexis F.},
     TITLE = {Global weak solutions for a {V}lasov-{F}okker-{P}lanck/{N}avier-{S}tokes system of equations},
   JOURNAL = {Math. Models Methods Appl. Sci.},
  FJOURNAL = {Mathematical Models and Methods in Applied Sciences},
    VOLUME = {17},
      YEAR = {2007},
    NUMBER = {7},
     PAGES = {1039--1063},
      ISSN = {0218-2025,1793-6314},
   MRCLASS = {35Q35 (35D05 76N10 82C22)},
  MRNUMBER = {2337430},
MRREVIEWER = {Beno\^{i}t\ P.\ Desjardins},
       DOI = {10.1142/S0218202507002194},
       URL = {https://doi.org/10.1142/S0218202507002194},
}

@article {DFT10,
    AUTHOR = {Duan, Renjun and Fornasier, Massimo and Toscani, Giuseppe},
     TITLE = {A kinetic flocking model with diffusion},
   JOURNAL = {Comm. Math. Phys.},
  FJOURNAL = {Communications in Mathematical Physics},
    VOLUME = {300},
      YEAR = {2010},
    NUMBER = {1},
     PAGES = {95--145},
      ISSN = {0010-3616,1432-0916},
   MRCLASS = {82C40 (35Q82 82C31)},
  MRNUMBER = {2725184},
MRREVIEWER = {Giuseppe\ Maria\ Coclite},
       DOI = {10.1007/s00220-010-1110-z},
       URL = {https://doi.org/10.1007/s00220-010-1110-z},
}

@article {GHMAZP10,
    AUTHOR = {Goudon, Thierry and He, Lingbing and Moussa, Ayman and Zhang, Ping},
     TITLE = {The {N}avier-{S}tokes-{V}lasov-{F}okker-{P}lanck system near equilibrium},
   JOURNAL = {SIAM J. Math. Anal.},
  FJOURNAL = {SIAM Journal on Mathematical Analysis},
    VOLUME = {42},
      YEAR = {2010},
    NUMBER = {5},
     PAGES = {2177--2202},
      ISSN = {0036-1410,1095-7154},
   MRCLASS = {35Q35 (35A01 35A09 35B10 35B40 76D05 82C31)},
  MRNUMBER = {2729436},
MRREVIEWER = {Gleb\ Germanovitch\ Doronin},
       DOI = {10.1137/090776755},
       URL = {https://doi.org/10.1137/090776755},
}

@article {YFSLY20,
    AUTHOR = {Su, Yunfei and Yao, Lei},
     TITLE = {Hydrodynamic limit for the inhomogeneous incompressible {N}avier-{S}tokes/{V}lasov-{F}okker-{P}lanck equations},
   JOURNAL = {J. Differential Equations},
  FJOURNAL = {Journal of Differential Equations},
    VOLUME = {269},
      YEAR = {2020},
    NUMBER = {2},
     PAGES = {1079--1116},
      ISSN = {0022-0396,1090-2732},
   MRCLASS = {35Q83 (35Q30 35Q84)},
  MRNUMBER = {4088464},
       DOI = {10.1016/j.jde.2019.12.027},
       URL = {https://doi.org/10.1016/j.jde.2019.12.027},
}

@Article{JNXCJZHJ18,
  author  = {Jiang, Ning and Xu, Chao-Jiang and Zhao, Huijiang},
  journal = {Indiana University Mathematics Journal},
  title   = {Incompressible {N}avier-{S}tokes-{F}ourier limit from the {B}oltzmann equation: classical solutions},
  year    = {2018},
  number  = {5},
  pages   = {1817-1855},
  volume  = {67},
  doi     = {https://doi.org/10.1512/iumj.2018.67.5940},
}

@article {JL22,
    AUTHOR = {Jiang, Ning and Luo, Yi-Long},
     TITLE = {From {V}lasov-{M}axwell-{B}oltzmann system to two-fluid
              incompressible {N}avier-{S}tokes-{F}ourier-{M}axwell system
              with {O}hm's law: convergence for classical solutions},
   JOURNAL = {Ann. PDE},
  FJOURNAL = {Annals of PDE},
    VOLUME = {8},
      YEAR = {2022},
    NUMBER = {1},
     PAGES = {Paper No. 4, 126},
      ISSN = {2524-5317,2199-2576},
   MRCLASS = {35Q30},
  MRNUMBER = {4382704},
MRREVIEWER = {Pedro\ Mar\'in Rubio},
       DOI = {10.1007/s40818-022-00117-6},
}

@article {Victory-Harold-Dwyer1990,
    Author = {Victory, Harold Dean and O'Dwyer, Brian P.},
     Title = {On classical solutions of {V}lasov-{P}oisson {F}okker-{P}lanck
              systems},
   Journal = {Indiana Univ. Math. J.},
  Fjournal = {Indiana University Mathematics Journal},
    Volume = {39},
      Year = {1990},
    Number = {1},
     Pages = {105--156},
      Issn = {0022-2518,1943-5258},
   Mrclass = {35Q99 (76X05 82D10)},
  Mrnumber = {1052014},
Mrreviewer = {Dressler,Klaus},
       Doi = {10.1512/iumj.1990.39.39009},
}

@article {Bouchut1993,
    Author = {Bouchut, Francois},
     Title = {Existence and uniqueness of a global smooth solution for the
              {V}lasov-{P}oisson-{F}okker-{P}lanck system in three
              dimensions},
   Journal = {J. Funct. Anal.},
  Fjournal = {Journal of Functional Analysis},
    Volume = {111},
      Year = {1993},
    Number = {1},
     Pages = {239--258},
      Issn = {0022-1236,1096-0783},
   Mrclass = {82D10 (35Q99 45K05 82C31)},
  Mrnumber = {1200643},
Mrreviewer = {Dressler,Klaus},
       Doi = {10.1006/jfan.1993.1011},
       Url = {https://doi.org/10.1006/jfan.1993.1011},
}

@article {Bouchut1995,
    Author = {Bouchut, Francois},
     Title = {Smoothing effect for the non-linear
              {V}lasov-{P}oisson-{F}okker-{P}lanck system},
   Journal = {J. Differential Equations},
  Fjournal = {Journal of Differential Equations},
    Volume = {122},
      Year = {1995},
    Number = {2},
     PAGES = {225--238},
      ISSN = {0022-0396,1090-2732},
   MRCLASS = {35Q99 (35K65 76X05 82D10)},
  MRNUMBER = {1355890},
       DOI = {10.1006/jdeq.1995.1146},
       URL = {https://doi.org/10.1006/jdeq.1995.1146},
}

@article {Carpio1998,
    AUTHOR = {Carpio, Ana},
     TITLE = {Long-time behaviour for solutions of the
              {V}lasov-{P}oisson-{F}okker-{P}lanck equation},
   JOURNAL = {Math. Methods Appl. Sci.},
  FJOURNAL = {Mathematical Methods in the Applied Sciences},
    VOLUME = {21},
      YEAR = {1998},
    NUMBER = {11},
     PAGES = {985--1014},
      ISSN = {0170-4214,1099-1476},
   MRCLASS = {82D10 (35Q99 76X05 82C31)},
  MRNUMBER = {1635981},
MRREVIEWER = {Glassey,R},
       DOI ={10.1002/(SICI)1099-1476(19980725)21:11<985::AID-MMA919>3.0.CO;2-B},
       URL ={https://doi.org/10.1002/(SICI)1099-1476(19980725)21:11<985::AID-MMA919>3.0.CO;2-B},
}

@article {Hwang-Jang2013,
    AUTHOR = {Hwang, Hyung Ju and Jang, Juhi},
     TITLE = {On the {V}lasov-{P}oisson-{F}okker-{P}lanck equation near
              {M}axwellian},
   JOURNAL = {Discrete Contin. Dyn. Syst. Ser. B},
  FJOURNAL = {Discrete and Continuous Dynamical Systems. Series B. A Journal
              Bridging Mathematics and Sciences},
    VOLUME = {18},
      YEAR = {2013},
    NUMBER = {3},
     PAGES = {681--691},
      ISSN = {1531-3492,1553-524X},
   MRCLASS = {82D10 (35Q84)},
  MRNUMBER = {3007749},
MRREVIEWER = {Pankavich，Stephen，D.},
       DOI = {10.3934/dcdsb.2013.18.681},
       URL = {https://doi.org/10.3934/dcdsb.2013.18.681},
}

@article {Tan-Fan2024,
    AUTHOR = {Tan, Lihua and Fan, Yingzhe},
     TITLE = {Global mild solutions to the relativistic
              {V}lasov-{P}oisson-{F}okker-{P}lanck system in the whole
              space},
   JOURNAL = {Chaos Solitons Fractals},
  FJOURNAL = {Chaos, Solitons \& Fractals},
    VOLUME = {183},
      YEAR = {2024},
     PAGES = {Paper No. 114882, 7},
      ISSN = {0960-0779,1873-2887},
   MRCLASS = {35A01 (76Y05)},
  MRNUMBER = {4741372},
       DOI = {10.1016/j.chaos.2024.114882},
       URL = {https://doi.org/10.1016/j.chaos.2024.114882},
}

@article {Poupaud-Soler2000,
    AUTHOR = {Poupaud, F. and Soler, J.},
     TITLE = {Parabolic limit and stability of the
              {V}lasov-{F}okker-{P}lanck system},
   JOURNAL = {Math. Models Methods Appl. Sci.},
  FJOURNAL = {Mathematical Models and Methods in Applied Sciences},
    VOLUME = {10},
      YEAR = {2000},
    NUMBER = {7},
     PAGES = {1027--1045},
      ISSN = {0218-2025,1793-6314},
   MRCLASS = {76X05 (45K05 82C21 82C22 82D10)},
  MRNUMBER = {1780148},
MRREVIEWER = {Glassey,R.},
       DOI = {10.1142/S0218202500000525},
       URL = {https://doi.org/10.1142/S0218202500000525},
}

@article {G2005,
    AUTHOR = {Goudon,Thierry},
     TITLE = {Hydrodynamic limit for the {V}lasov-{P}oisson-{F}okker-{P}lanck system: analysis of the
              two-dimensional case},
   JOURNAL = {Math. Models Methods Appl. Sci.},
  FJOURNAL = {Mathematical Models and Methods in Applied Sciences},
    VOLUME = {15},
      YEAR = {2005},
    NUMBER = {5},
     PAGES = {737--752},
      ISSN = {0218-2025,1793-6314},
   MRCLASS = {82C31 (45K05 82D10)},
  MRNUMBER = {2139941},
MRREVIEWER = {Bal,Guillaume},
       DOI = {10.1142/S021820250500056X},
       URL = {https://doi.org/10.1142/S021820250500056X},
}

@article {Zhong2022,
    AUTHOR = {Zhong, Mingying},
     TITLE = {Diffusion limit and the optimal convergence rate of the
              {V}lasov-{P}oisson-{F}okker-{P}lanck system},
   JOURNAL = {Kinet. Relat. Models},
  FJOURNAL = {Kinetic and Related Models},
    VOLUME = {15},
      YEAR = {2022},
    NUMBER = {1},
     PAGES = {1--26},
      ISSN = {1937-5093,1937-5077},
   MRCLASS = {76P05 (35Q60 82C40 82D10)},
  MRNUMBER = {4389615},
MRREVIEWER = {Marzia\ Bisi},
       DOI = {10.3934/krm.2021041},
       URL = {https://doi.org/10.3934/krm.2021041},
}

@article{B2023,
    AUTHOR = {Blaustein, Alain},
     TITLE = {Diffusive limit of the {V}lasov-{P}oisson-{F}okker-{P}lanck
              model: quantitative and strong convergence results},
   JOURNAL = {SIAM J. Math. Anal.},
  FJOURNAL = {SIAM Journal on Mathematical Analysis},
    VOLUME = {55},
      YEAR = {2023},
    NUMBER = {5},
     PAGES = {5464--5482},
      ISSN = {0036-1410,1095-7154},
   MRCLASS = {35Q53 (35B40 35Q84 45K05)},
  MRNUMBER = {4649398},
MRREVIEWER = {Asadzadeh,Mohammad},
       DOI = {10.1137/22M1530197},
       URL = {https://doi.org/10.1137/22M1530197},
}

@article {Herda-Rodrigues2018,
    AUTHOR = {Herda, Maxime and Rodrigues, L. Miguel},
     TITLE = {Large-time behavior of solutions to {V}lasov-{P}oisson-{F}okker-{P}lanck equations: from evanescent collisions to diffusive limit},
   JOURNAL = {J. Stat. Phys.},
  FJOURNAL = {Journal of Statistical Physics},
    VOLUME = {170},
      YEAR = {2018},
    NUMBER = {5},
     PAGES = {895--931},
      ISSN = {0022-4715,1572-9613},
   MRCLASS = {35Q83 (35B30 35B35 35B40 35H10 35Q84 82C31)},
  MRNUMBER = {3767000},
       DOI = {10.1007/s10955-018-1963-7},
       URL = {https://doi.org/10.1007/s10955-018-1963-7},
}

@article {Lehman-Negulescu2025,
    AUTHOR = {Lehman, Etienne and Negulescu, Claudia},
     TITLE = {Vlasov-{P}oisson-{F}okker-{P}lanck equation in the adiabatic
              asymptotics},
   JOURNAL = {Commun. Math. Sci.},
  FJOURNAL = {Communications in Mathematical Sciences},
    VOLUME = {23},
      YEAR = {2025},
    NUMBER = {3},
     PAGES = {669--710},
      ISSN = {1539-6746,1945-0796},
   MRCLASS = {65M22 (35Q84 65M70 65M75)},
  MRNUMBER = {4861796},
MRREVIEWER = {Sun,Yajuan},
       DOI = {10.4310/cms.250208214950},
       URL = {https://doi.org/10.4310/cms.250208214950},
}

@article {El-Masmoudi2010,
    AUTHOR = {El Ghani, Najoua and Masmoudi, Nader},
     TITLE = {Diffusion limit of the {V}lasov-{P}oisson-{F}okker-{P}lanck
              system},
   JOURNAL = {Commun. Math. Sci.},
  FJOURNAL = {Communications in Mathematical Sciences},
    VOLUME = {8},
      YEAR = {2010},
    NUMBER = {2},
     PAGES = {463--479},
      ISSN = {1539-6746,1945-0796},
   MRCLASS = {82D10 (35Q82 45K05)},
  MRNUMBER = {2664460},
       DOI = {10.4310/cms.2010.v8.n2.a9},
       URL = {https://doi.org/10.4310/cms.2010.v8.n2.a9},
}

@article {Addala-Dolbeault-Li-Tayeb2021,
    AUTHOR = {Addala, Lanoir and Dolbeault, Jean and Li, Xingyu and Tayeb,
              M. Lazhar},
     TITLE = {{$L^2$}-hypocoercivity and large time asymptotics of
              the linearized {V}lasov-{P}oisson-{F}okker-{P}lanck system},
   JOURNAL = {J. Stat. Phys.},
  FJOURNAL = {Journal of Statistical Physics},
    VOLUME = {184},
      YEAR = {2021},
    NUMBER = {1},
     PAGES = {Paper No. 4, 34},
      ISSN = {0022-4715,1572-9613},
   MRCLASS = {82C40 (35H10 35P15 47G20 82D10)},
  MRNUMBER = {4277286},
MRREVIEWER = {Zheng,Yuxi},
       DOI = {10.1007/s10955-021-02784-4},
       URL = {https://doi.org/10.1007/s10955-021-02784-4},
}

@article {El2010,
    AUTHOR = {El Ghani, Najoua},
     TITLE = {Diffusion limit for the {V}lasov-{M}axwell-{F}okker-{P}lanck
              system},
   JOURNAL = {IAENG Int. J. Appl. Math.},
  FJOURNAL = {IAENG International Journal of Applied Mathematics},
    VOLUME = {40},
      YEAR = {2010},
    NUMBER = {3},
     PAGES = {159--166},
      ISSN = {1992-9978,1992-9986},
   MRCLASS = {35Q84 (35B65)},
  MRNUMBER = {2732472},
}

@article {Dong-Yang-Zhong2019,
    AUTHOR = {Dong, Hongjie and Yang, Tong and Zhong, Mingying},
     TITLE = {Exterior problem of the linear {V}lasov-{P}oisson-{B}oltzmann
              system},
   JOURNAL = {SIAM J. Math. Anal.},
  FJOURNAL = {SIAM Journal on Mathematical Analysis},
    VOLUME = {51},
      YEAR = {2019},
    NUMBER = {3},
     PAGES = {1792--1823},
      ISSN = {0036-1410,1095-7154},
   MRCLASS = {76P05 (35Q20 35Q83 82C40 82D05)},
  MRNUMBER = {3947289},
       DOI = {10.1137/18M1232875},
       URL = {https://doi.org/10.1137/18M1232875},
}

@article {Dong-Yastrzhembskiy2024,
    AUTHOR = {Dong, Hongjie and Yastrzhembskiy, Timur},
     TITLE = {Global {$L_p$} estimates for kinetic
              {K}olmogorov-{F}okker-{P}lanck equations in divergence form},
   JOURNAL = {SIAM J. Math. Anal.},
  FJOURNAL = {SIAM Journal on Mathematical Analysis},
    VOLUME = {56},
      YEAR = {2024},
    NUMBER = {1},
     PAGES = {1223--1263},
      ISSN = {0036-1410,1095-7154},
   MRCLASS = {35K70 (35B45 35H10)},
  MRNUMBER = {4704640},
MRREVIEWER = {Aleksi\'c,Jelena},
       DOI = {10.1137/22M1512120},
       URL = {https://doi.org/10.1137/22M1512120},
}

@article {Dong-Guo-Ouyang2022,
    AUTHOR = {Dong, Hongjie and Guo, Yan and Ouyang, Zhimeng},
     TITLE = {The {V}lasov-{P}oisson-{L}andau system with the
              specular-reflection boundary condition},
   JOURNAL = {Arch. Ration. Mech. Anal.},
  FJOURNAL = {Archive for Rational Mechanics and Analysis},
    VOLUME = {246},
      YEAR = {2022},
    NUMBER = {2-3},
     PAGES = {333--396},
      ISSN = {0003-9527,1432-0673},
   MRCLASS = {35Q83 (82D10)},
  MRNUMBER = {4514055},
       DOI = {10.1007/s00205-022-01818-9},
       URL = {https://doi.org/10.1007/s00205-022-01818-9},
}

@article {Dong-Yastrzhembskiy2022,
    AUTHOR = {Dong, Hongjie and Yastrzhembskiy, Timur},
     TITLE = {Global {$L_p$} estimates for kinetic
              {K}olmogorov-{F}okker-{P}lanck equations in nondivergence
              form},
   JOURNAL = {Arch. Ration. Mech. Anal.},
  FJOURNAL = {Archive for Rational Mechanics and Analysis},
    VOLUME = {245},
      YEAR = {2022},
    NUMBER = {1},
     PAGES = {501--564},
      ISSN = {0003-9527,1432-0673},
   MRCLASS = {35K70},
  MRNUMBER = {4444079},
MRREVIEWER = {Anceschi,Francesca},
       DOI = {10.1007/s00205-022-01786-0},
       URL = {https://doi.org/10.1007/s00205-022-01786-0},
}

@article {Dong-Guo-Ouyang-Yastrzhembskiy2024,
    AUTHOR = {Dong, Hongjie and Guo, Yan and Ouyang, Zhimeng and
              Yastrzhembskiy, Timur},
     TITLE = {The local well-posedness of the relativistic
              {V}lasov-{M}axwell-{L}andau system with the specular
              reflection boundary condition},
   JOURNAL = {SIAM J. Math. Anal.},
  FJOURNAL = {SIAM Journal on Mathematical Analysis},
    VOLUME = {56},
      YEAR = {2024},
    NUMBER = {5},
     PAGES = {6613--6688},
      ISSN = {0036-1410,1095-7154},
   MRCLASS = {35Q83 (34A12 35H10 35K70 35Q61 35Q84)},
  MRNUMBER = {4802874},
MRREVIEWER = {Consiglieri,Luisa  da Cunha e Costa},
       DOI = {10.1137/23M1608938},
       URL = {https://doi.org/10.1137/23M1608938},
}

@article {Guo2002,
    AUTHOR = {Guo, Yan},
     TITLE = {The {V}lasov-{P}oisson-{B}oltzmann system near {M}axwellians},
   JOURNAL = {Comm. Pure Appl. Math.},
  FJOURNAL = {Communications on Pure and Applied Mathematics},
    VOLUME = {55},
      YEAR = {2002},
    NUMBER = {9},
     PAGES = {1104--1135},
      ISSN = {0010-3640,1097-0312},
   MRCLASS = {82C40 (35A05 35L60 35Q60 45K05 82D10)},
  MRNUMBER = {1908664},
MRREVIEWER = {Glassey.R.},
       DOI = {10.1002/cpa.10040},
       URL = {https://doi.org/10.1002/cpa.10040},
}

@article {Yang-Zhao2006,
    AUTHOR = {Yang, Tong and Zhao, Huijiang},
     TITLE = {Global existence of classical solutions to the
              {V}lasov-{P}oisson-{B}oltzmann system},
   JOURNAL = {Comm. Math. Phys.},
  FJOURNAL = {Communications in Mathematical Physics},
    VOLUME = {268},
      YEAR = {2006},
    NUMBER = {3},
     PAGES = {569--605},
      ISSN = {0010-3616,1432-0916},
   MRCLASS = {35F20 (35Q35 35Q60 76N10 76P05 76X05 82D10)},
  MRNUMBER = {2259207},
       DOI = {10.1007/s00220-006-0103-4},
       URL = {https://doi.org/10.1007/s00220-006-0103-4},
}

@article {Yang-Yu-Zhao2006,
    AUTHOR = {Yang, Tong and Yu, Hongjun and Zhao, Huijiang},
     TITLE = {Cauchy problem for the {V}lasov-{P}oisson-{B}oltzmann system},
   JOURNAL = {Arch. Ration. Mech. Anal.},
  FJOURNAL = {Archive for Rational Mechanics and Analysis},
    VOLUME = {182},
      YEAR = {2006},
    NUMBER = {3},
     PAGES = {415--470},
      ISSN = {0003-9527,1432-0673},
   MRCLASS = {35F20 (76P05 76X05 82C40)},
  MRNUMBER = {2276498},
       DOI = {10.1007/s00205-006-0009-5},
       URL = {https://doi.org/10.1007/s00205-006-0009-5},
}

@article {Jiang-Lei-Zhao2024,
    AUTHOR = {Jiang, Ning and Lei, Yuanjie and Zhao, Huijiang},
     TITLE = {On the {V}lasov-{P}oisson-{B}oltzmann limit of the
              {V}lasov-{M}axwell-{B}oltzmann system},
   JOURNAL = {J. Funct. Anal.},
  FJOURNAL = {Journal of Functional Analysis},
    VOLUME = {287},
      YEAR = {2024},
    NUMBER = {7},
     PAGES = {Paper No. 110529, 110},
      ISSN = {0022-1236,1096-0783},
   MRCLASS = {35Q60 (35Q20 35Q83 76X05)},
  MRNUMBER = {4756502},
MRREVIEWER = {Paquet,Luc},
       DOI = {10.1016/j.jfa.2024.110529},
       URL = {https://doi.org/10.1016/j.jfa.2024.110529},
}

@article {Xiao-Xiong-Zhao2017,
    AUTHOR = {Xiao, Qinghua and Xiong, Linjie and Zhao, Huijiang},
     TITLE = {The {V}lasov-{P}oisson-{B}oltzmann system for the whole range
              of cutoff soft potentials},
   JOURNAL = {J. Funct. Anal.},
  FJOURNAL = {Journal of Functional Analysis},
    VOLUME = {272},
      YEAR = {2017},
    NUMBER = {1},
     PAGES = {166--226},
      ISSN = {0022-1236,1096-0783},
   MRCLASS = {82D10 (35Q20 35Q83)},
  MRNUMBER = {3567504},
       DOI = {10.1016/j.jfa.2016.09.017},
       URL = {https://doi.org/10.1016/j.jfa.2016.09.017},
}

@article {Li-Yang-Zhong2016,
    AUTHOR = {Li, Hai-Liang and Yang, Tong and Zhong, Mingying},
     TITLE = {Spectrum analysis and optimal decay rates of the bipolar
              {V}lasov-{P}oisson-{B}oltzmann equations},
   JOURNAL = {Indiana Univ. Math. J.},
  FJOURNAL = {Indiana University Mathematics Journal},
    VOLUME = {65},
      YEAR = {2016},
    NUMBER = {2},
     PAGES = {665--725},
      ISSN = {0022-2518,1943-5258},
   MRCLASS = {35Q83 (35B40 35Q20 35Q82 82D10)},
  MRNUMBER = {3498181},
MRREVIEWER = {Bertrand\ Lods},
       DOI = {10.1512/iumj.2016.65.5730},
       URL = {https://doi.org/10.1512/iumj.2016.65.5730},
}

@article {WLL2015,
    AUTHOR = {Wu, Hao and Lin, Tai-Chia and Liu, Chun},
     TITLE = {Diffusion limit of kinetic equations for multiple species
              charged particles},
   JOURNAL = {Arch. Ration. Mech. Anal.},
  FJOURNAL = {Archive for Rational Mechanics and Analysis},
    VOLUME = {215},
      YEAR = {2015},
    NUMBER = {2},
     PAGES = {419--441},
      ISSN = {0003-9527,1432-0673},
   MRCLASS = {82D10 (35Q83 82C31 92C37)},
  MRNUMBER = {3294407},
       DOI = {10.1007/s00205-014-0784-3},
       URL = {https://doi.org/10.1007/s00205-014-0784-3},
}

@article {Guo2012,
    AUTHOR = {Guo, Yan},
     TITLE = {The {V}lasov-{P}oisson-{L}andau system in a periodic box},
   JOURNAL = {J. Amer. Math. Soc.},
  FJOURNAL = {Journal of the American Mathematical Society},
    VOLUME = {25},
      YEAR = {2012},
    NUMBER = {3},
     PAGES = {759--812},
      ISSN = {0894-0347,1088-6834},
   MRCLASS = {35Q83 (35B10)},
  MRNUMBER = {2904573},
MRREVIEWER = {Tadmon,Calvin},
       DOI = {10.1090/S0894-0347-2011-00722-4},
       URL = {https://doi.org/10.1090/S0894-0347-2011-00722-4},
}

@article {Duan-Yu2020,
    AUTHOR = {Duan, Renjun and Yu, Hongjun},
     TITLE = {The {V}lasov-{P}oisson-{L}andau system near a local
              {M}axwellian},
   JOURNAL = {Adv. Math.},
  FJOURNAL = {Advances in Mathematics},
    VOLUME = {362},
      YEAR = {2020},
     PAGES = {106956, 83},
      ISSN = {0001-8708,1090-2082},
   MRCLASS = {35Q83 (35A01 35B35 35B40)},
  MRNUMBER = {4050580},
       DOI = {10.1016/j.aim.2019.106956},
       URL = {https://doi.org/10.1016/j.aim.2019.106956},
}

\end{document}